\documentclass[11pt]{article}

\usepackage[T1]{fontenc}
\usepackage{lmodern}
\usepackage[toc,page]{appendix}
\usepackage{amsmath, amssymb, amsthm}
\usepackage[margin=2cm]{geometry}
\usepackage{color, graphicx}
\usepackage{multirow}
\usepackage{parskip}

\usepackage{etoolbox}
\makeatletter
\patchcmd{\@maketitle}{\LARGE \@title}{\LARGE\bfseries\@title}{}{}

\renewcommand{\@seccntformat}[1]{\csname the#1\endcsname.\quad}
\makeatother

\definecolor{darkblue}{rgb}{0,0,.5}

\usepackage{hyperref}
\hypersetup{
	colorlinks=true,		
	linkcolor=darkblue,		
	citecolor=darkblue,		
	urlcolor=darkblue 		
}

\makeatletter
\def\th@plain{%
	\thm@notefont{}
	\itshape 
}
\def\th@definition{%
	\thm@notefont{}
	\normalfont 
}

\renewenvironment{proof}[1][\proofname]{\par
	\normalfont
	\topsep0\p@\@plus3\p@ \trivlist
	\item[\hskip\labelsep\itshape
	#1\@addpunct{.}]\ignorespaces
}{%
	\qed\endtrivlist
}
\makeatother

\usepackage[capitalize, nameinlink, noabbrev]{cleveref}

\newtheorem{theorem}{Theorem}[section]
\newtheorem{lemma}[theorem]{Lemma}

\newtheorem{proposition}[theorem]{Proposition}
\theoremstyle{definition}
\newtheorem{definition}[theorem]{Definition}
\newtheorem{example}[theorem]{Example}
\newtheorem{remark}[theorem]{Remark}

\usepackage[shortlabels]{enumitem}

\usepackage{empheq}
\definecolor{myblue}{rgb}{.8, .8, 1}

\usepackage[numbers,sort&compress]{natbib}
\bibpunct[, ]{[}{]}{,}{n}{,}{,}
\newcommand{\vp}{\varepsilon}

\newcommand{\Rb}{\R\cup \{+\infty\}}
\newcommand{\ball}{\B}
\def\F{\mathcal{F}}  
\def\<{\langle}
\def\>{\rangle}

\def\vol{\textnormal{Vol}\,}
\newcommand{\ballc}{\overline{\ball}}

\def\inti{\textnormal{int}\,}
\def\cl{\textnormal{cl}\,}
\def\epi{\textnormal{epi}\,}

\def\gph{\textnormal{gph}\,}
\def\dom{\textnormal{Dom}\,}

\def\proj{\textnormal{proj}\,}
\newcommand{\E}{\mathcal E}
\newcommand{\R}{\mathbb R}
\newcommand{\N}{\mathbb N}
\newcommand{\B}{\mathbb B}
\def\dom{\mathop{\rm dom\,}}

\def\clco{\mathop{\rm clco\,}}

\def\cl{\mathop{\rm cl\,}}

\def\argmin{\mathop{\rm argmin\,}}

\begin{document}

\title{Geometric Stability Analysis for Differential Inclusions Governed by Maximally Monotone Operators\thanks{The research of HS \& MT has been partially supported by Gulf University for Science and Technology and the Research Center (CAMB) under project code: ISG – Case No. 160. }
\thanks{The research of MND was partially supported by the Australian Research Council Discovery Project DP230101749.}}

\author{
Hassan Saoud\thanks{Department of Mathematics and Natural Sciences \& Center of Applied Mathematics and Bioinformatics (CAMB), Gulf University for Science and Technology, P.O. Box 7207, Hawally 32093, Kuwait. Email: \texttt{saoud.h@gust.edu.kw}.}, 
~
Michel Th\'era\thanks{Laboratoire XLIM UMR-CNRS, 7252, Universit\'e de Limoges, 87032, Limoges, France. Email: \texttt{michel.thera@unilim.fr},}
~~and~
Minh N. Dao\thanks{School of Sciences, RMIT University, Melbourne, VIC 3000, Australia. Email: \texttt{minh.dao@rmit.edu.au}.}
}

\date{}

\maketitle

\abstract{This paper develops a geometric framework for the stability analysis of differential inclusions governed by maximally monotone operators. A key structural decomposition expresses the operator as the sum of a convexified limit mapping and a normal cone. However, the resulting dynamics are often difficult to analyze directly due to the absence of Lipschitz selections and boundedness. To overcome these challenges, we introduce a regularized system based on a fixed Lipschitz approximation of the convexified mapping. From this approximation, we extract a single-valued Lipschitz selection that preserves the essential geometric features of the original system. This framework enables the application of nonsmooth Lyapunov methods and Hamiltonian-based stability criteria. Instead of approximating trajectories, we focus on analyzing a simplified system that faithfully reflects the structure of the original dynamics. Several examples are provided to illustrate the method's practicality and scope. The analysis is carried out under a uniform boundedness assumption on the convexified limit mapping.}


{\small

}

\section{Introduction}\label{sec1}

Lyapunov functions are essential for analyzing differential equations, especially in stability theory. Identifying  Lyapunov functions is crucial for both theoretical and practical applications.
The work focuses on the stability of a differential inclusion involving a maximally monotone operator. Let $A: \R^n \rightrightarrows \R^n$  be a maximally  monotone operator. Consider a locally Lipschitz  function $f$ defined on  $\cl(\dom A)$, the closure of $\dom A$. Given $x_0 \in \cl(\dom A)$, we consider the following dynamical system 
\begin{align}\label{eq:prob} 
\left\{\begin{array}{ll}\dot{x}(t) \in f(x(t)) - A(x(t)) 
& \textrm{ a.e. }\,\, t\in [0,+\infty) \\
x(0)=x_0 \in \cl(\dom A).
\end{array}\right.
\tag{$\mathcal{DI}$}
\end{align}
It is well-known that the system \eqref{eq:prob} admits a unique solution for all $t\geq0$, see for example \cite{barbu, Brezis1973}. This paper aims to establish the stability of sets of the problem \eqref{eq:prob},
with a particular focus on \emph{pointwise asymptotic stability of a set $(\mathbf{PAS})$.}
Pointwise asymptotic stability means that every point in the set that is an equilibrium is Lyapunov stable, and every solution starting near the set converges and ends up within the set.
To clarify the notion of pointwise asymptotic stability of a set, it is essential to note that in the literature, $\mathbf{PAS}$  is typically defined with respect to the set of equilibria (see \cite{goebel10, goebel16, goebel18}). Some references  use the term \emph{semistability} to describe $\mathbf{PAS}$  in relation to the set of equilibria (see, for example, \cite{bhat99, bhat03, hui09, hui10}). In this paper, we use the term $\mathbf{PAS}$  when discussing an arbitrary set, while we use the term semistability when focusing on the set of equilibria. Therefore, an equilibrium is semistable if it is Lyapunov stable, and every solution starting in a neighborhood of the equilibrium converges to a (possibly different) Lyapunov stable equilibrium. It is worth noting that semistability does not mean that  the set of equilibria is  asymptotically stable. In fact, a trajectory can converge to the equilibria set without converging to any specific equilibrium point. Semistability, however,  does not automatically mean that the equilibrium set is asymptotically stable in a straightforward manner. This arises because stability of sets  is defined with respect to distance, especially when dealing with noncompact sets, which is the case for the system \eqref{eq:prob}. Therefore, semistability and  set stability of the equilibrium set are two separate concepts. In the case  where the set of equilibrium is a singleton, then the semistability is equivalent to the stability of this set. This stability concept is suitable for the cases involving nonisolated equilibria and it has been examined within the framework of both differential equations \cite{bhat99, bhat03},  and differential inclusions \cite{hui09, hui10}. Moreover, in \cite{bhat10}, the authors provide sufficient conditions for semistability  using arc-length-based Lyapunov criteria. Studies in \cite{goebel10, goebel16, goebel18} extensively analyse  semistability for both hybrid systems and difference inclusions, providing sufficient conditions  in terms of set-valued Lyapunov functions. In \cite{saoud15}, the notion of semistability was extended to differential inclusions, where the operator $A$ is the subdifferential of a proper lower semicontinuous convex function. The results are expressed in terms of continuously differentiable Lyapunov functions.

As previously mentioned, the aim of this paper is to investigate the pointwise asymptotic stability of a set, and consequently, the semistability (of the set of equilibria). All the results are given based on \emph{Lyapunov pairs} approach  associated to the differential inclusion \eqref{eq:prob} and the \emph{lower Hamiltonian} corresponding to the set-valued mapping $f-A$. We will demonstrate our capability to identify a set $S$ depending on  dynamics and  Lyapunov pairs ensuring the system $(S, f-A)$ is invariant. This set is paramount to prove the $\mathbf{PAS}$  and the semistability of the set of equilibria as well. It is important to note that we are not providing a detailed characterization of invariant sets in this context. If $x(\cdot)$ is solves \eqref{eq:prob} starting at $x_0 \in \cl(\dom A)$ and $V, W$  are lower semicontinuous extended real-valued functions, then $(V,W)$ is called \emph{Lyapunov pair} for \eqref{eq:prob} if the function below
is decreasing along the solution of \eqref{eq:prob}: 
\[
t \mapsto V(x(t)) + \displaystyle{\int_0^t W(x(\tau))\,d\tau}.
\]
Hence, if $W\equiv 0$ we say that the function $V$ is a Lyapunov function of \eqref{eq:prob}. 

Extensive research in recent decades has explored invariant sets via Lyapunov pairs. In \cite{ clarke1995,clarke1997approximate,clarke2}, the authors studied the classical case of  differential inclusions of the form 
\[
\dot{x}(t) \in F(x(t)). 
\]   
Here, the set-valued mapping $F$ is a CUSCO (convex upper semicontinuous, nonempty and compact valued), and  it is further  assumed to satisfy a certain linear growth condition. For a closed set $S$, establishing  the invariance of the system $(S,F)$, the  authors in \cite{ clarke1995,clarke2} introduce a proximal criterion. This criterion is  given in terms of the lower Hamiltonian corresponding to $F$ using  an \emph{Euler solution} of the inclusion and it requires that the set-valued mapping $F$ must be locally Lipschitz.  Furthermore,  in \cite{donchev2005}, the authors extended these  invariance results to cover cases involving  one-side Lipschitz time-dependent set-valued mappings, which are less restrictive compared to Lipschitz set-valued mapping. In \cite{clarke1997approximate} and under the same assumptions on $F$, the authors provide necessary and sufficient conditions for a subset $S$ to be approximately invariant with respect to  approximate solutions of the differential inclusion. This concept of approximate invariance generalizes the classical invariance concept and it is based on the concept of  $\varepsilon-$trajectory corresponding the differential inclusion.

The initial and classical characterization of the Lyapunov pairs for differential inclusions of type \eqref{eq:prob}  was presented by Pazy in \cite{pazy1981, pazy1981lyapunov}, who considered the system 
 \[
 \dot{x}(t) \in -A(x(t)), 
 \]   
providing criteria in terms of directional-like derivatives using the Moreau-Yosida approximation of the operator $A$. 
Pazy's results were extended to  system \eqref{eq:prob} by \cite{kocan2002} and
\cite{carjua2009}. The characterization of Lyapunov pairs 
is given in \cite{kocan2002} 
through the associated Hamilton-Jacobi partial differential equations, with solutions interpreted in the viscosity sense. Their method also establishes a new adequate condition for Lyapunov pairs, generalizing the results in \cite{pazy1981, pazy1981lyapunov}. The results given in \cite{carjua2009}  offer a distinct and more explicit characterization of Lyapunov pairs  for \eqref{eq:prob} without relying on viscosity solutions. Their approach  is achieved through the contingent derivative associated with the operator. The proof utilizes tangency and flow-invariance arguments, complemented by a-priori estimates and approximation techniques. Note that in both \cite{kocan2002, carjua2009} the results are based on implicit criteria that are significantly dependent on the semi-group generated by the maximally monotone operator $A$. 

Given that the operator $A$ is not explicitly known and considering that all previous results are based on determining its associated semi-group, 
more and better conditions that are independent of this semi-group are needed. 
Following \cite{clarke2,clarke1997approximate,clarke1995}, the authors in \cite{adly2018invariant, Adly1} offer alternative criteria for characterizing invariant sets  under the differential inclusion \eqref{eq:prob}  by characterizing the Lyapunov pairs. This approach relies solely on the data $A$ and $f$, eliminating the necessity to explicitly solve the equation. Contrary to the classical differential inclusion case,  the right-hand side in \eqref{eq:prob} might be empty, non-compact, or unbounded and potentially not outer semicontinuous. Boundedness of the operator $A$ (or the boundedness of the minimal norm mapping $A^\circ$) and on the approximate invariance technique introduced in \cite{clarke1997approximate} are at the base of those criteria.

To overcome the difficulties posed by unboundedness of the right-hand side of \eqref{eq:prob} and to avoid complex assumptions on $A$, this paper relies on 
properties of the maximally monotone operator $A$, the inclusion \eqref{eq:prob} and its solution. The approach is guided by two main ideas. First, taking advantage of the operator properties on both the interior and the boundary of its domain, we will split the operator $A$ into the sum of two set-valued mappings: one that is continuous and the other that represents the normal cone (in the sense of convex analysis) to the closure of its domain, such that 
\(
A = F + N_{\cl(\dom A)}.
\)
Therefore, problem \eqref{eq:prob} can be equivalently expressed as 
\[ 
\left\{\begin{array}{ll}\dot{x}(t) \in f(x(t)) - F(x(t))-N_C(x(t)) 
& \textrm{ a.e. }\,\, t\in [0,+\infty), \\
x(0)=x_0 \in C:=\cl(\dom A).
\end{array}\right.
\]
Second, we extend the main results of \cite{Thibault2013}, established for an $r$-prox-regular set $C$. Specially, at the solution $x(t)$ of \eqref{eq:prob}, for $v(t) \in F(x(t))$ and for $\eta(t) := -\dot{x}(t) + f(x(t)) - v(t) \in N_C (x(t))$, the estimation 
\(
\Vert \eta(t) \Vert \leq \Vert f(x(t)) - v(t) \Vert 
\)
plays an important role in the proof of our main results. The proposed method depends on choosing an appropriate selection from the set-valued mapping $F(x)$, preferably Lipschitz continuous. Unlike traditional regularization techniques that approximate the solution trajectory of a nonsmooth system, the present approach is structural in nature. Rather than recovering the original solution \( x(t) \) through limiting procedures, we replace the nonsmooth operator \( F \) with a fixed Lipschitz continuous approximation \( F_k \) that contains \(F \) and retains key geometric features of the original operator. From this approximation, we extract a single-valued Lipschitz selection \( \psi_k \), and perform the stability analysis on the resulting regularized system. The resulting robust Lyapunov-based analysis does not require convergence in \( k \) or reconstruction of original trajectories, but remains compatible with classical tools from nonsmooth analysis. 

The role of the regularized systems is purely auxiliary. The stability notions studied in this paper -pointwise asymptotic stability and semistability- are geometric, characterized through Hamiltonian inequalities and Lyapunov-type estimates involving admissible directions and distance functions, rather than through the approximation of individual trajectories. The Lipschitz outer approximations and single-valued selections are introduced solely as technical tools to establish intermediate inequalities that, once verified, result in stability properties for the original differential inclusion.

\textbf{Contribution and novelty.}
Unlike classical sweeping process frameworks, where the normal cone term constitutes the primary source of nonsmoothness, the present work addresses differential inclusions governed by general maximally monotone operators. The normal cone $N_{\cl(\dom A)}$ arises only through a structural decomposition of $A$, while the main analytical difficulty lies in the treatment of the convexified limit mapping $F = \clco A_0$. The stability and semistability results are therefore obtained at the level of Hamiltonian inequalities associated with this decomposition, using fixed Lipschitz outer regularizations solely as auxiliary analytical tools.

As mentioned earlier, the goal of this paper is to examine the $\mathbf{PAS}$  and the semistability of the dynamic \eqref{eq:prob}. All conditions will be presented in terms of the lower Hamiltonian corresponding to the set-valued mapping $f-F-N_C$ via nonsmooth Lyapunov pairs $(V,W)$ and will involve nonsmooth analysis tools and techniques. In fact, since we are focused on the set convergence and, more specifically, with the distance function, we provide a geometric approach based on proximal analysis and its differentiability properties. We apply techniques similar to those used in \cite{radulescu1997geometric} for approximating \emph{horizontal} normals to the epigraph of the lower semicontinuous function $V$ with \emph{non-horizontal} ones. Contrary to \cite{adly2018invariant,kocan2002}, the technical condition 
\(
\forall x\in \dom V,\quad V(x) = \liminf_{y \xrightarrow{C}x}{V(y)}
\)
is no longer needed, which means our result is based on minimal assumptions on the function $V$. The analysis is conducted under a uniform boundedness assumption on the convexified limit mapping, which is essential for constructing Lipschitz regularizations with controlled constants and for avoiding additional regularity assumptions on the dynamics.

The work is organized as follows. Section~\ref{sec:notation} introduces the notation and preliminary concepts from nonsmooth analysis and monotone operator theory. Section~\ref{sec:ingredients} presents the analytical ingredients needed for the proofs of the main results. In particular,
Subsection~\ref{sec:maxmonotone} analyzes the structure of maximally monotone operators and introduces the geometric decomposition that underlies the proposed framework. Subsection~\ref{sec:selection} is devoted to the construction of a Lipschitz continuous selection associated with the convexified limit mapping. Subsection~\ref{sec:stability-framework} develops the stability framework and Hamiltonian analysis. It introduces the relevant stability notions, establishes the Hamiltonian inequalities, and presents the approximation tools required for the proofs. Section~\ref{sec:proofs-main-results} contains the proofs of the main results stated in Section~\ref{sec:main-results}. Section~\ref{application} illustrates the theory through several applications, including smooth inertial systems and nonsmooth differential inclusions, showing how the assumptions and Hamiltonian conditions can be verified in concrete settings.
Section~\ref{sec:conclusion} concludes the paper. The appendix provides a detailed analysis of the Lipschitz approximation procedure used in Example~\ref{ex:approx} and collects additional technical arguments supporting the main results.
\section{Notation and Preliminaries}
\label{sec:notation}
Throughout, $\R^n$  is the $n$ dimensional Euclidean space with inner product  
$\langle\cdot,\cdot\rangle$ and induced norm $\Vert \cdot\Vert$,  i.e., for all $x \in \R^n$, {$\Vert x\Vert := \sqrt{\langle x,x\rangle}$}. We denote by $\B(x,r)$ (respectively, $\ballc (x,r)$)  the open (respectively, the closed) ball in $\R^n$ with center $x$ and radius $r$, and we set $\B:=\B(0,1)$ for the unit ball. We denote by \(\mathbb{S}^{n-1}\) the unit sphere in \(\R^n\). The \emph{indicator function} of $S$ is the function $I_S$ taking the values $0$ on $S$ and $+\infty$ off $S$.

Let $\varphi: \R^n \to \R\cup\{+\infty\}$ be an extended-real-valued function. The \emph{(effective) domain}, the \emph{epigraph} and the \emph{lower level set} of $\varphi$ are defined by  
\[
\dom \varphi:=\{x\in \R^n:\, \varphi(x)<+\infty\},\,\, \epi\varphi:=\{(x,\alpha)\in
 \R^n\times\R:\, \varphi(x)\leq\alpha\},
 \]
\[
\;\;\text{and}\;\; [\varphi \leq \alpha]_{\mid{\dom \varphi}} :=\{x\in
 \dom\varphi:\, \varphi(x)\leq\alpha\}.
\]
We say that $\varphi$ is \emph{proper} if $\dom %
\varphi\neq \varnothing$ and that $\varphi$ is \emph{convex} if $\epi\varphi$ is convex. 

Let us recall that $\varphi$ is \emph{lower semicontinuous}
(l.s.c., for short) at $y \in \R^n$ if for every $\alpha \in \R$ with $\varphi(y) > \alpha$, there is $\delta > 0$ such that 
\[
\forall x\in \ball(y,\delta),\quad \varphi(x) > \alpha.
\] 
We simply say that $\varphi$ is l.s.c. if it is l.s.c. at 
every point of $\R^n$. Equivalently, $\varphi$ is l.s.c. if and only if its epigraph is closed.
We denote by $\mathcal{F}(\R^n)$ (resp. $\F^+(\R^n)
$) the set of extended-real-valued, proper and lower semicontinuous functions (resp. nonnegative). 
For a convex function $\varphi \in \mathcal{F}(\R^n)$ and for $x\in \dom \varphi$, a vector $\zeta \in \R^n$ is a \emph{subgradient} of $\varphi$ at $x$ if for all $y \in \R^n$, we have
$$\varphi(y) \geq 
\varphi(x) + \big\langle \zeta \,, y-x \big\rangle. $$
The \emph{Fenchel subdifferential} of $\varphi$ at $x$ collects all subgradients and is denoted by  $\partial\varphi (x)$.

We proceed by giving some definitions and results from \emph{nonsmooth analysis}. The basic references for these notions and facts can be found in details in \cite{clarke1, clarke2, Rock}.  
Let $\varphi$ be a function of $\F(\R^n)$ and let $x\in \dom \varphi$. We say that a vector $\zeta \in \R^n$ is a \emph{proximal subgradient} of $\varphi$ at $x$ if there exist $\eta > 0$ and $\sigma \geq 0$ such that 
$$\forall y \in \ball(x, \eta),\quad \varphi(y) \geq 
\varphi(x) + \big\langle \zeta \,, y-x \big\rangle - \sigma \Vert y-x\Vert
^{2}.$$
The \emph{proximal subdifferential} of $\varphi$ at $x$ collects all proximal subgradients and is denoted by  $\partial_P \varphi (x)$. The set $\partial_P \varphi (x)$ is convex, possibly empty and not necessarily closed.

Let $S$ be a nonempty and closed subset of $\R^n$ and let  $x$ be a point not lying in $S$.  A point $z \in S$ is called \emph{closest point} or \emph{projection of $x$ onto $S$}, denoted by $\proj_S (x)$, if and only if $\{z\} \subseteq S \cap \ballc (x; \Vert x-z \Vert)$ and  $S \cap \ball (x; \Vert x-z \Vert) = \varnothing$. In general, the projection $\proj_S(x)$ may be multivalued; it is single-valued
when the set $S$ is closed and convex.
In addition, $z \in \proj_S(x)$ if and only if for all $s \in [0,1]$, $z \in \proj_S (z+s(x-z))$. The collection of vectors in the form $s(x - z)$, where $s \geq 0$, is referred to as \emph{proximal normal cone to $S$}. This set can also be described in the following manner
   $N_S^{P}(x) := \partial_P I_S(x).$
If the set $S$ is convex, then $N_S^{P}(x)$ is denoted by $N_S(x)$ where 
 $N_S(x) := \partial I_S(x).$
Moreover, a geometric characterization of the notion of proximal subdifferential, 
 previously defined, is given by the following 
$$\zeta \in \partial_{P} \varphi (x) \iff (\zeta,-1)\in N_{\epi \varphi}^{P}(x, \varphi(x)).$$ 
In the case where the function $\varphi$ is also convex, we can have the same characterization of the normal cone in terms of the Fenchel subdifferential of $\varphi$.

Finally, we provide some useful results related to the proximal subgradients of the distance function $\mathbf{d} (\cdot;S)$ associated to a nonempty closed subset $S$. For more details, readers can refer to \cite{clarke2}.   
\begin{proposition}[Theorem~2.6 in \cite{clarke2}]
\label{proxdist}
Let $\emptyset\neq S\subset\R^n$ be a closed set, $x\notin S$, and $z\in \proj_S (x)$. Then, 
    \(\mbox{for all $s \in (0,1)$,}\quad 
    \partial_P\mathbf{d}\big(z+s(x-z);S\big) = \Big\{ \displaystyle{\frac{x-z}{\Vert x-z \Vert}}\Big\}.
\hfill\qed    \)
\end{proposition}
\begin{theorem}[Theorem~2.6 in \cite{clarke2}]
    \label{mvi}
       Given $x,y \in \R^n,$ for all $r < \mathbf{d}(y;S) - \mathbf{d}(x;S)$ and $\vp > 0$, there exist $z \in [x,y] +\vp\ball$  and $\zeta \in \partial_P \mathbf{d}(z;S)$ such that 
       \(
       r < \big\langle \zeta,y-x\big\rangle.
       \)
\end{theorem}
We now introduce the concept of local prox-regularity for sets. For a more comprehensive discussion on local prox-regularity, refer to \cite{Poliquin2000LocalDO,Thibaultbook}. 
\begin{definition}
    \label{proxreg}
    For positive real numbers $r$ and $\alpha$, the closed set $S$ is said to be $(r,\alpha)$\emph{-prox-regular} at a point $\bar{x}  \in S$ provided that for any $x \in S \cap \ball (\bar{x},\alpha)$ and any $v \in N^P_S(x)$ such that $\Vert v \Vert < r$, one has $x = proj_S(x + v)$. 
    The set $S$ is $r$\emph{-prox-regular} (resp. \emph{prox-regular}) at $\bar{x}$ when it is $(r,\alpha)$-prox-regular at $\bar{x}$ for some real $\alpha > 0$  (resp. for some numbers $r > 0$ and $\alpha > 0$). The set $S$ is said to be $r$-\emph{uniformly prox-regular} when $\alpha = +\infty$.
\end{definition}
Uniformly prox-regular sets include and strictly contain convex sets. Thus, every closed and convex set is $r$-uniformly prox-regular for any $r \geq 0$.

If $\varphi: \R^n \to \R$ is locally Lipschitz and if $\mathcal{N}$ be any subset of zero measure in $\R^n$,
and if $\mathcal{N}_\varphi$ be the set of points in $\R^n$ at which $\varphi$ fails to be differentiable, we define the \emph{Clarke subdifferential} of $\varphi$ at $x\in \dom \varphi$ as
\[
 \partial_C \varphi(x) = \clco \{\lim_{i \to +\infty} \nabla \varphi(x_i), x_i \to x, x_i \notin \mathcal{N}, x_i \notin \mathcal{N}_\varphi\}.
\]
In addition to nonsmooth analysis results, important properties related to maximally monotone operators are needed to provide a more comprehensive understanding of \eqref{eq:prob}. A multifunction $A: \R^n \rightrightarrows \R^n$ is
 said to be \emph{monotone} if 
\[
\forall (y_1,y_2) \in Ax_1 \times Ax_2,\quad \langle y_1 -
 y_2,x_1 - x_2 \rangle \geq 0. 
\]
The \emph{domain} of $A$ is the set 
\(
\dom A = \big\{x\in \R^n :\,  A(x) \neq \varnothing \big\}.
\)
A monotone operator $A$ is \emph{maximally monotone} 
provided its graph given by  $\gph A = \big\{(x,y)\;: y\in A(x)\big\},$  
cannot be properly enlarged without destroying monotonicity.  

Unlike its closure,  the domain of a  maximally  monotone operator may not be closed and convex (it is nearly convex; see {\cite{Rock}}). Its  values  are closed and convex but may be unbounded or even empty. A typical  example of  maximally  monotone operator is the Fenchel 
subdifferential of an extended-real-valued lower 
semicontinuous and convex function $\varphi$. We have 
\(
\dom (\partial \varphi) \subseteq \dom \varphi \subseteq
\cl(\dom \varphi) = \cl(\dom \partial\varphi).
\)
A maximally monotone operator $A$ is locally bounded at $x$ if and only if $x\in \inti(\dom A)$, see \cite{Rock1,phelps}. The Fenchel subdifferential of a  
{proper l.s.c. convex function} is locally 
bounded on the interior of its domain. When applied to  the indicator function  of a convex closed set $C$, this implies that operator  $N_C$ is locally bounded on $\inti C$. We refer the reader to standard references such as \cite{Bauschke, Rock} for more details.
\section{Main Results}
\label{sec:main-results}
Throughout \eqref{eq:prob} is considered together with its equivalent formulation
\[
\dot{x}(t) \in f(x(t)) - \clco A_0(x(t)) - N_C(x(t)),
\qquad C := \cl(\dom A),
\]
obtained from a structural decomposition of the maximally monotone operator $A$.

As stated in the Introduction, the objective of this work is to investigate \emph{pointwise asymptotic stability} $(\mathbf{PAS})$ and \emph{semistability} $(\mathbf{SS})$ for problem~\eqref{eq:prob}. The main analytical difficulty arises from the presence of the convexified limit mapping $\clco A_0$, which is in general only outer semicontinuous and not Lipschitz continuous, preventing a direct application of classical Lyapunov and invariance techniques.

Our approach is geometric in nature and relies on Hamiltonian inequalities associated with the set-valued mapping
$f - \clco A_0 - N_C$. Rather than approximating individual trajectories, the analysis is carried out at the level of Lyapunov functions and Hamiltonians, using Lipschitz outer approximations of $\clco A_0$ as auxiliary analytical tools. These approximations allow for a controlled approximation of the Hamiltonian, denoted by $\boldsymbol{h}$, while preserving the essential structure of the original dynamics, making it possible to transfer stability and invariance properties from the regularized systems to the original differential inclusion.

Under a uniform boundedness assumption on $\clco A_0$ and a local boundedness assumption on $f$, this framework yields invariance, convergence, and semistability results for problem~\eqref{eq:prob}.

Our first stability result establishes a Lyapunov decrease and an invariance principle for sublevel sets of lower semicontinuous functions.
\begin{theorem}[Invariance and decrease]
\label{t:invariance}
Suppose that~\eqref{eq:uniform-boundedness} and ~\eqref{bound} hold, and there exists a function $V \in \mathcal{F}(\R^n)$ such that, for all $x \in C$, 
\begin{align}\label{eq:H1}
\boldsymbol{h}\big(x,\partial_P V (x)\big) \leq 0.
\tag{$\mathcal{H}1$}
\end{align}
Let $x(t)$ be a solution of \eqref{eq:prob}. Then the following statements hold:
\begin{enumerate}
\item\label{t:invariance_decrease}
The pair $\big(V,f-\clco A_0-N_C\big)$ is decreasing.
\item\label{t:invariance_main} 
There exists an $\alpha > 0$  such that  $\big([V\leq\alpha]_{\mid{\dom V}} ,f-\clco A_0-N_C\big)$ is invariant.
\end{enumerate}
\end{theorem}
The next result provides a convergence criterion based on a Lyapunov pair and yields pointwise asymptotic stability of the associated zero set.
\begin{proposition}[Convergence and pointwise asymptotic stability]
\label{ConvW}
Suppose that~\eqref{eq:uniform-boundedness} and \eqref{bound} hold, and there exists $W\in \F^+ (\R^n)$ such that, for all $x \in C$, 
    \begin{align}\label{eq:H2}
    \boldsymbol{h}\big(x,\partial_P V (x)\big) \leq -W(x).
    \tag{$\mathcal{H}2$}
    \end{align}
Then, for every $x_0$ there exists a solution $x(\cdot)$ of  \eqref{eq:prob} with $x(0) = x_0$ such that 
\[
\lim_{t\to+\infty}\mathbf d(x(t);W^{-1}(0))=0.
\]
If each point in $W^{-1}(0)$ exhibits Lyapunov stability, then $W^{-1}(0)$ qualifies as a $\mathbf{PAS}$.
\end{proposition}
The final result is on semistability of the equilibrium set associated with~\eqref{eq:prob}.
\begin{proposition}
\label{ConvE}
Suppose that condition $(\mathcal{A}_1)$ from Definition \ref{PAS} is satisfied for the set $\mathcal{E}$  and that the solution of \eqref{eq:prob} is bounded. In addition, suppose that for a given $x_0 \in C$, $\displaystyle\lim_{t \to +\infty} \mathbf{d}(x(t);\mathcal{E}) = 0.$ Then $x(t) \to z$ where $z \in \mathcal{E}$.
\end{proposition}
\begin{theorem}[Semistability]
\label{SS}
Suppose that the conditions outlined in Theorem \ref{t:invariance} are met. If, in addition, $\mathcal{E} \subseteq W^{-1}(0)$ and every point of $W^{-1}(0)$ is Lyapunov stable, then \eqref{eq:prob} is $\mathbf{SS}$. 
\end{theorem}
The proofs of the above results rely on two key analytical ingredients. The first one concerns the structure of the velocity of solutions to \eqref{eq:prob}. In particular, Lemma~\ref{Thibault}, which is a consequence of \cite[Theorem~3.1]{Thibault2013} adapted to the case where the set $C$ is convex, provides a precise decomposition of the velocity $\dot{x}(t)$ into components belonging to $\clco A_0(x(t))$ and $N_C(x(t))$, together with a crucial estimate on the normal component. This result plays a fundamental role in controlling the dynamics near the boundary of $C$ and in establishing the Lipschitz regularity of solutions.

The second ingredient is the approximation of the lower Hamiltonian associated with the original dynamics. Using Lipschitz outer approximations $F_k$ of $\clco A_0$, Proposition~\ref{prop:hamiltonian-approx} establishes a quantitative relation between the original Hamiltonian $\boldsymbol{h}$ and its regularized counterpart $\boldsymbol{h}_k$. This estimate ensures that Hamiltonian inequalities verified for the regularized system can be transferred to the original differential inclusion,
without requiring convergence of trajectories.

These two tools combine geometric arguments, Lyapunov pairs, and Hamiltonian inequalities to derive invariance, convergence, and semistability properties for problem~\eqref{eq:prob}. Proofs of all the results stated in this section are given in Section~\ref{sec:proofs-main-results}.
\section{Analytical Ingredients for the Main Results}
\label{sec:ingredients}
\subsection{Structure of Maximally Monotone Operators}
\label{sec:maxmonotone}
Maximally monotone operators exhibit distinct behaviors in the interior and at the boundary of their domain, and understanding their connection provides valuable insights into their structure. To explore these properties, consider a maximally monotone operator $A: \R^n \rightrightarrows \R^n$ with $\inti(\dom A)$ nonempty, and let $E$ be the subset of $\inti(\dom A)$ on which $A$ is single-valued. Define the mapping \( A_0 : \R^n \rightrightarrows \R^n \) by
\begin{align}
\label{eq:A0}
A_0(x) = \{ v: \exists\, (x_k)\subset E \text{~with~} x_k\to x \text{~and~} A(x_k) \to v \}.    
\end{align}
According to \cite[Theorem 12.67]{Rock}, \( A_0 \) is single-valued on \( E \) and agrees there with \( A \); moreover, $\dom A = \dom A_0 \subseteq \cl E = \cl(\dom A)$ and, for all $x \in \R^n$,
\begin{align}
\label{operator}
A(x) = \clco A_0(x) + N_{\cl(\dom A)}(x).
\end{align}
where $\clco A_0(x)$ means the closed convex hull of $ A_0(x)$.
It is known that \( A \) is continuous on \( E \), the set where it is single-valued and that the set of points where \( A \) is differentiable is a dense subset of \( \dom A \) contained in \( E \). Thus, \( A_0 \) is single-valued and continuous on \( E \), coinciding with \( A \) on this set. In addition,  \( A \equiv A_0 \) is locally bounded on \( \inti(\dom A) \). Furthermore, 
\(
\cl E = \cl(\dom A_0) = \cl(\dom A).
\)

Decomposition \eqref{operator} is better understood by examining regularity properties of \( x \mapsto \clco A_0(x) \), namely its outer semicontinuity and local boundedness on \( \cl(\dom A) \).
\begin{definition}[Outer semicontinuity {\cite[Definition~5.4]{Rock}}]
A set-valued mapping $F:\mathbb{R}^n \rightrightarrows \mathbb{R}^m$ is said to be \emph{outer semicontinuous} at $\bar x$ if
\(
\limsup_{x\to\bar x} F(x) \subseteq F(\bar x).
\)
The mapping $F$ is said to be outer semicontinuous if it is outer semicontinuous at every point.
\end{definition}
\begin{proposition}[Outer semicontinuity and local boundedness of \( \clco A_0 \)]
\label{prop:usc-locbnd}
Let \( A : \R^n \rightrightarrows \R^n \) be maximally monotone with \( \inti(\dom A) \) nonempty, and \( A_0 \) be given in~\eqref{eq:A0}. Then the set-valued mapping \( x \mapsto \clco A_0(x) \) is outer semicontinuous and locally bounded on \( \cl(\dom A) \).
\end{proposition}
\begin{proof}
We begin by establishing the closedness of the graph of \( A_0 \). Let \( x_k \to x \in \cl(E) \) and \( v_k \in A_0(x_k) \) with \( v_k \to v \). By the definition of \( A_0 \), for each \( k \), there exists a sequence \( (y_{k,\ell})_{\ell \in \N} \subset E \) such that
\(
y_{k,\ell} \to x_k \quad \text{and} \quad A(y_{k,\ell}) \to v_k \quad \text{as } \ell \to +\infty.
\)

For each \( k \), select an index \( \ell_k \in \N \) such that
\(
\|y_{k,\ell_k} - x_k\| < \tfrac{1}{k}, \quad \text{and} \quad \|A(y_{k,\ell_k}) - v_k\| < \tfrac{1}{k},
\)
and define \( y_k := y_{k,\ell_k} \in E \). Then
\begin{align*}
\|y_k - x\| &\leq \|y_k - x_k\| + \|x_k - x\| \to 0, \textnormal{ and }\\
\|A(y_k) - v\| & \leq \|A(y_k) - v_k\| + \|v_k - v\| \to 0.
\end{align*}
Hence, \( y_k \to x \) and \( A(y_k) \to v \), so \( v \in A_0(x) \). Therefore, the graph of \( A_0 \) is closed.

Since each value \( \clco A_0(x) \) is nonempty, convex, and closed, and since the operation \( x \mapsto \clco A_0(x) \) preserves graph closedness, it follows that the mapping \( x \mapsto \clco A_0(x) \) also has a closed graph. By a classical result in variational analysis (see \cite[Theorem~5.7]{Rock}), a set-valued mapping with nonempty closed values is outer semicontinuous if and only if its graph is closed. Thus, \( x \mapsto \clco A_0(x) \) is outer semicontinuous on \( \cl(E) = \cl(\dom A) \).

To show local boundedness, observe that \( A \equiv A_0 \) is continuous and locally bounded on \( E \subseteq \inti(\dom A) \). Since \( A_0 \) has a closed graph and is locally bounded on the dense set \( E \), it is also locally bounded on \( \cl(E) \). Fix \( x \in \cl(E) \). Then there exists a neighborhood \( U \) of $x$ and a constant \( M > 0 \) such that
\(
\bigcup_{y \in U} A_0(y) \subseteq \ball(0, M).
\)
As convex hulls of bounded sets are bounded, for all $y \in U$,
\(
\clco A_0(y) \subseteq \clco(\ball(0, M)) = \ball(0, M),
\)
and therefore
\(
\bigcup_{y \in U} \clco A_0(y) \subseteq \ball(0, M),
\)
which proves that \( \clco A_0 \) is locally bounded at \( x \). Since \( x \in \cl(E) \) was arbitrary, local boundedness holds on \( \cl(\dom A) \).
\end{proof}
The expression \eqref{operator} can be applied to the case where \( A = \partial \varphi\), the Fenchel subdifferential of an extended-real-valued, lower semicontinuous, and convex function \( \varphi \). More precisely, since the convex function \( \varphi \) is locally Lipschitz on \( \inti(\dom \varphi) \), we deduce that \(\clco A_0 = \partial_C \varphi \) which coincides with the closed convex hull of \( \nabla \varphi \) on the differentiable region of \( \varphi \). Therefore, the decomposition becomes
\begin{align*}
\partial \varphi(x) = \partial_C \varphi(x) + N_{\cl(\dom \varphi)}(x)
\end{align*}
(see also \cite[Theorem~25.6]{Rock2}). Moreover, since \( \varphi \) is locally Lipschitz on \( \inti(\dom \varphi) \), its Clarke subdifferential \( \partial_C \varphi \) is locally bounded on this set.
\subsection{Construction of a Lipschitz Continuous Selection }
\label{sec:selection}
The goal of this section is to construct a Lipschitz continuous selection from the set-valued mapping
\[
x \mapsto \clco A_0(x),
\]
which, as previously established, is outer semicontinuous with nonempty, closed, and convex values on \( \cl(\dom A) \). However, these regularity properties are not sufficient to apply the approximation result we rely on--namely, \cite[Theorem~2.5]{smirnov2002}--which requires that the values of the mapping be uniformly bounded over the domain.

In general, the mapping \( x \mapsto \clco A_0(x) \) is not uniformly bounded on \( \cl(\dom A) \), that is,
\[
\bigcup_{x \in \cl(\dom A)} \clco A_0(x)
\]
may be unbounded. To overcome this limitation, we assume that
\begin{equation}
\label{eq:uniform-boundedness}
\text{there exists~} b > 0 \text{ such that, for all } x \in \cl(\dom A),\ \clco A_0(x) \subseteq b \B, 
\end{equation}
where \( \B \) denotes the closed unit ball in \( \R^n \).
\begin{remark}
    Assumption~\eqref{eq:uniform-boundedness} is restrictive from the viewpoint of
general maximally monotone operators. It is nevertheless essential for the regularization strategy adopted in this work, as it allows the construction of Lipschitz approximations and single-valued selections with \emph{uniformly controlled Lipschitz constants}.
Under weaker boundedness or regularity assumptions, standard regularization procedures typically lead to selections whose Lipschitz constants \emph{blow up} when approaching the boundary of $\dom A$, which is not sufficient for the geometric and selection-based analysis developed here.
\end{remark}
The following result will be used in our analysis.
\begin{theorem}
[{\cite[Theorem~2.5]{smirnov2002}}]
\label{smirnov}
Let \( F : \R^n \rightrightarrows \R^n \) be an outer semicontinuous set-valued mapping with closed convex values. Suppose that there exists \( b > 0 \) such that,  for all \( x \in \R^n \), \( F(x) \subseteq b \B \). Then there exists a sequence of locally Lipschitz set-valued mappings \( F_k : \R^n \rightrightarrows \R^n \), \( k \in \mathbb{N} \), satisfying the following conditions:
\begin{enumerate}
    \item For every \( x \in \R^n \),
    \[
    F(x) \subseteq \cdots \subseteq F_{k+1}(x) \subseteq F_k(x) \subseteq \cdots \subseteq F_0(x) \subseteq b \B;
    \]
    \item For every \( \varepsilon > 0 \) and \( x \in \R^n \), there exists an integer \( k(\varepsilon, x) \) such that
    \[
    F_k(x) \subseteq F(x) + \varepsilon \B \quad \text{whenever } k > k(\varepsilon, x).
    \]
\end{enumerate}
\end{theorem}
This result provides the foundation for constructing a regularized selection that approximates \( \clco A_0 \) while possessing enhanced regularity. For each fixed \( k \), the corresponding approximation \( F_k \) admits a locally Lipschitz continuous selection, which can be used as a smooth replacement for \( \clco A_0 \) in the analysis of regularized dynamics or stability properties. Since \( \clco A_0(x) \subseteq F_k(x) \) for every \( x \in \R^n \), the original differential inclusion
\[
\dot{x}(t) \in f(x(t)) - \clco A_0(x(t)) - N_{\cl(\dom A)}(x(t))
\]
can be replaced by the relaxed system
\[
\dot{x}(t) \in f(x(t)) - F_k(x(t)) - N_{\cl(\dom A)}(x(t)).
\]
This substitution is not intended to approximate the trajectories of the original system, but rather to define a regularized inclusion that preserves the geometric structure while enjoying enhanced regularity. In particular, the local Lipschitz continuity of \( F_k \) allows for the construction of a single-valued, locally Lipschitz selection. This results in a differential inclusion with a well-defined and more manageable right-hand side, as shown in the following lemma.

By applying Theorem~\ref{smirnov} to the approximation \( F_k \) of
\( \clco A_0 \), we associate with each fixed \( k \) a canonical
single-valued selection defined via the Steiner point.
More precisely, we set
\[
\psi_k(x) := s\bigl(F_k(x)\bigr), \qquad x \in \cl(\dom A),
\]
where \( s \) denotes the Steiner point mapping recalled in~\ref{appendixA}.

Since each \( F_k(x) \) is nonempty, convex, compact, and uniformly
bounded by \( b \B \), the Steiner selection \( \psi_k \) is well defined and satisfies \( \psi_k(x) \in F_k(x) \) for all
\( x \in \cl(\dom A) \). Moreover, because \( F_k \) is locally Lipschitz continuous with respect to the Hausdorff--Pompeiu distance, Theorem~\ref{thm:steiner-aubin} and Lemma~\ref{lem:steiner-selection} (~\ref{appendixA}) ensure that the mapping \( x \mapsto \psi_k(x) \) is locally Lipschitz on \( \cl(\dom A) \). In addition, since \(s\bigl(F_k(x)\bigr)\in F_k(x)\) and \(F_k(x)\subset b\B\), we have
$\|\psi_k(x)\|\le b$ for all $x \in \cl(\dom A).$

As a consequence, the regularized vector field
$x \longmapsto f(x) - \psi_k(x)$ is locally Lipschitz on \( \cl(\dom A) \). Therefore, for every initial condition in \( \cl(\dom A) \), the regularized differential inclusion
\[
\dot{x}(t) \in f(x(t)) - \psi_k(x(t)) - N_C(x(t))
\]
admits a unique solution.

Recall that throughout the analysis, the regularized inclusion involving \( F_k \) provides a structural replacement that preserves the essential geometric properties of the original system, without aiming to approximate its trajectories.
\begin{example}
\label{ex:approx}
Let \( \varphi : \R^n \to \R \cup \{+\infty\} \) be a proper lower semicontinuous convex function, and let \( A := \partial \varphi \) be its subdifferential. As mentioned after Proposition~\ref{prop:usc-locbnd}, we have \( \clco A_0(x) = \partial_C \varphi(x) \). Suppose that \( \clco A_0 \) is uniformly bounded on \( \cl(\dom \varphi) \), i.e., there exists \( b > 0 \) such that, for all \( x \in \cl(\dom \varphi) \),
\begin{equation}
\label{eq:boun-clarke}
\partial_C \varphi(x) \subseteq b \B.
\end{equation}
Then, by Theorem~\ref{smirnov}, there exists a sequence of set-valued mappings \( (F_k) \) with nonempty, convex, compact values and locally Lipschitz graphs such that, for all \( x \in \cl(\dom \varphi) \),
\[
\partial_C \varphi(x) \subseteq F_k(x) \subseteq b \B .
\]
By applying the Steiner selection to \( F_k \), we define
\[
\psi_k(x) := s\bigl(F_k(x)\bigr), \qquad x \in \cl(\dom \varphi),
\]
where \( s \) denotes the Steiner point mapping recalled in~\ref{appendixA}. Note that condition~\eqref{eq:boun-clarke} is automatically satisfied if \(\varphi \) is globally Lipschitz on \( \dom \varphi \).

We construct an approximation \( F_k : \R^n \rightrightarrows \R^n \),
indexed by a parameter \( \delta_k := 1/k \), that is convex-valued,
uniformly bounded, and Lipschitz continuous with respect to the
Hausdorff--Pompeiu distance. It is defined by
\[
F_k(x) :=
\begin{cases}
\{ u(x) \}, & \text{if } \|v(x)\| \ge \delta_k, \\[8pt]
(1 - \alpha(x))\,\B + \alpha(x)\{ u(x) \},
& \text{if } 0 < \|v(x)\| < \delta_k, \\[8pt]
N_C(x) \cap \B, & \text{if } x \in C,
\end{cases}
\]
with \( \alpha(x) := \dfrac{\|v(x)\|}{\delta_k} \in (0,1) \).

We associate with this approximation the single-valued selection
defined via the Steiner point,
\[
\psi_k(x) := s\bigl(F_k(x)\bigr), \qquad x \in \R^n,
\]
where \( s \) denotes the Steiner point mapping recalled in~\ref{appendixA}.
For this particular construction, the Steiner selection admits the
explicit representation
\[
\psi_k(x) =
\begin{cases}
u(x), & \text{if } \|v(x)\| \ge \delta_k, \\[10pt]
\alpha(x)\,u(x), & \text{if } 0 < \|v(x)\| < \delta_k, \\[10pt]
s\bigl(N_C(x) \cap \B\bigr), & \text{if } x \in C.
\end{cases}
\]
In particular, \( \psi_k(x) \in F_k(x) \) for all \( x \in \R^n \).
Moreover, \( \psi_k \) is globally defined, uniformly bounded by \(1\),
and Lipschitz continuous.

We now consider a special case of the previous construction by taking the distance function to the origin,
\( \varphi(x) = \|x\| \), which corresponds to the case where the set \( C = \{0\} \). The structure of the approximation follows identically, with \( v(x) = x \), \( u(x) = \dfrac{x}{\|x\|} \), and \( \alpha(x) = \dfrac{\|x\|}{\delta_k} \).

We refer to~\ref{appendixB} for the detailed construction of the approximating mappings \( F_k \) and further discussion of the selection \( \psi_k \), including alternative explicit realizations available in this particular geometric setting.
\end{example}
\subsection{Stability Framework and Hamiltonian Analysis}
\label{sec:stability-framework}
Given a maximally  monotone operator $A$ and a Lipschitz continuous function $f$ defined on $\cl(\dom A)\subseteq \R^n$,  
we consider again  the inclusion given by \eqref{eq:prob}. From the expression $\eqref{operator}$ with $C:= \cl(\dom A)$, system \eqref{eq:prob} is given by
\[ 
\left\{\begin{array}{ll}\dot{x}(t) \in f(x(t)) - \clco A_0 (x(t)) - N_{C}(x(t))
& \textrm{ a.e. }\,\, t\in [0,+\infty) \\
x(0)=x_0 \in C.
\end{array}\right.
\]
For $x_0 \in C$, it is known that there exists a unique  absolutely continuous function  $x(\cdot): [0,+\infty) \to \R^n$ such that $x(\cdot)$ satisfies \eqref{eq:prob} (see \cite{brezis71,barbu}).

In this section, we present the stability definitions needed for developing the main results of this paper. We begin by the Lyapunov stability theory for system \eqref{eq:prob}. We denote by $x(t)$ the solution of  \eqref{eq:prob}. Moreover, let $\E$  be the set of equilibrium points associated to \eqref{eq:prob} given by 
\[
\E := \{ \bar{x} \in C: f(\bar{x}) \in \clco A_0 (\bar{x}) + N_C (\bar{x}) \}.
\]
\begin{definition}[Lyapunov stability and asymptotic stability]
\label{stability}
\begin{enumerate}
\item An equilibrium point $\bar{x} \in \E$ is \emph{Lyapunov stable} if for every $\varepsilon > 0$ there exists $\delta = \delta (\varepsilon) > 0$ such that for $\Vert x_0 - \bar{x} \Vert \leq \delta$, the solution $x(t)$ of \eqref{eq:prob} with $x(0) = x_0$ satisfies, for all $t \geq 0$, $\Vert x(t) - \bar{x} \Vert < \varepsilon$.
\item  An equilibrium point $\bar{x} \in \E$ is \emph{asymptotically stable $(\mathbf{AS})$}
if it is Lyapunov stable and attractive, i.e., there exists $\delta>0$ such that
for every $x_0$ with $\|x_0-\bar{x}\|<\delta$, the corresponding solution $x(t)$
satisfies $\lim_{t\to+\infty} x(t)=\bar{x}$.
\end{enumerate}    
\end{definition}
\begin{definition}[Pointwise asymptotic stable]
\label{PAS}
A set $\mathcal{Z}\subseteq C$ is said to be \emph{pointwise asymptotically stable
$(\mathbf{PAS})$} for \eqref{eq:prob} if
\begin{enumerate}
\item[($\mathcal{A}_1$)] every $z\in\mathcal{Z}$ is Lyapunov stable;
\item[($\mathcal{A}_2$)] there exists $\delta>0$ such that for every initial condition
$x_0\in C$ satisfying $d(x_0,\mathcal{Z})<\delta$, the corresponding solution $x(t)$
of \eqref{eq:prob} is convergent and satisfies
\[
\lim_{t\to+\infty} x(t)\in\mathcal{Z}.
\]
\end{enumerate}

\end{definition}
In Definition~\ref{PAS}, the set $\mathcal Z$ is arbitrary; however, the condition $(\mathcal A_1)$ restricts pointwise asymptotic stability to those sets whose elements are equilibrium points that are Lyapunov stable.
Pointwise asymptotic stability means that every solution starting sufficiently near $\mathcal Z$ converges, with its limit belonging to $\mathcal Z$. While $\mathbf{PAS}$ coincides with asymptotic stability when $\mathcal Z$ reduces to a single equilibrium, it differs significantly for noncompact
equilibrium sets. In that case, $\mathbf{PAS}$ should be understood as a \emph{pointwise convergence property} rather than as a distance-based stability of the set: it does not impose uniform convergence toward $\mathcal Z$ nor convergence to a distinguished equilibrium, but only guarantees convergence of each trajectory to some Lyapunov stable point in $\mathcal Z$, possibly depending on the initial condition. This behavior is typical of systems with nonisolated equilibria and motivates the use of $\mathbf{PAS}$ and semistability instead of classical set stability
notions.
\begin{definition}[Semistability]
An equilibrium $\bar{x} \in \mathcal{E}$ is said to be \emph{semistable  $(\mathbf{SS})$} if it satisfies ($\mathcal{A}_1$) and ($\mathcal{A}_2$) for $\mathcal{Z} = \mathcal{E}$.
\end{definition}
It is evident that for an equilibrium, asymptotic stability leads to semistability, which in turn leads to Lyapunov stability. It is worth noting that semistability is not equivalent to asymptotic stability of the equilibrium set, since a trajectory may approach the set without converging to a specific equilibrium point, especially when the set is noncompact.
\begin{example}
\begin{enumerate}
\item If $\mathcal{Z}$ is a singleton or, more general, a compact set, then $\mathbf{PAS}$  implies $\mathbf{AS}$.
\item
Consider the steepest descent dynamics
\[
\dot{x}(t) \in -\partial \varphi(x(t)),    
\]
where $\varphi$ is a convex function. Then $\mathcal{Z} = \argmin \varphi$ (the set of equilibrium points).
\end{enumerate}
\end{example}
As illustrated in Definition \ref{stability}, the requirement for establishing the system's stability involves having an explicit solution to the system. To address this challenge, we turn to the non-direct Lyapunov method. The core principle of Lyapunov's original idea for verifying the stability of a dynamical system entails the search for an associated nonnegative real-valued function. Indeed, a lower semicontinuous function $V: C \to \R\cup\{+\infty\}$ is called a Lyapunov function for system \eqref{eq:prob} if for every $x_0 \in \dom V$, the function $V(x(t))$ is non-increasing as a function of $t$. Observe that, $\dom V$ denotes the effective domain of the extended-real-valued  function $V$ defined on $C$ i.e. 
\[
\dom V = \big\{ x\in C: V(x) < +\infty \big\}.
\]
Let $V\in \F(\R^n)$ and let $W \in \F^+(\R^n)$ defined on $C$. We define the \emph{Lyapunov pair} $(V,W)$ for the system \eqref{eq:prob} if, for all $x\in C$ and all $t\geq 0$, 
\begin{align}\label{lyapair}
V(x(t)) + \displaystyle\int_0^t W(x(s)) \leq V(x_0).
\end{align}
It is important to recall that, if $W(x(t))$ is nonnegative and globally Lipschitz on $\R^+$ with $\displaystyle\int_0^t W(x(\tau)) d\tau$ is bounded for $t\geq 0$, then $W(x(t))$ goes to $0$ as $t \to +\infty$.

It is evident that $V$ qualifies as a Lyapunov function if and only if the pair $(V, 0)$ meets the criteria for being a Lyapunov pair. In this case, the system $\big(V,f-\clco A_0-N_C\big)$ is said to be \emph{decreasing}.

In studying the stability of dynamical systems, invariance theory is a key concept, alongside the Lyapunov method. 
\begin{definition}
    Given a closed set $S$. We say that the pair $\big(S,f-\clco A_0 -N_C\big)$ is  \emph{invariant} if for all $x_0 \in S \cap C$, the  trajectory $x(\cdot)$ starting from $x(0) = x_0$, which is uniquely defined on the interval $[0,\infty)$ remains within the set $S$  for all $t \geq 0.$
\end{definition}
We aim to study the behavior of trajectories of \eqref{eq:prob} as they approach a closed set \( S \subseteq \R^n \). Given a solution \( x(t) \) and a point \( z \in \proj_S(x(t)) \), the inequality
\[
\langle f(x(t)) - g(t) - \eta(t),\ x(t) - z \rangle \leq 0
\]
with \( g(t) \in\clco A_0(x(t)) \) and \( \eta(t) \in N_C(x(t)) \) indicates that the velocity points toward the set \( S \).

Although this condition provides geometric intuition, the absence of a Lipschitz selection from \(\clco A_0 \) limits direct analysis. To address this, we consider the regularized inclusion
\[
\dot{x}(t) \in f(x(t)) - F_k(x(t)) - N_C(x(t)),
\]
where \( F_k \) is a Lipschitz outer approximation of \( \clco A_0 \),
constructed under the boundedness assumption~\eqref{eq:uniform-boundedness} via Theorem~\ref{smirnov}. The corresponding Steiner selection $\psi_k(x) := s\bigl(F_k(x)\bigr)$
is well defined, uniformly bounded, and locally Lipschitz continuous.

This regularization allows us to apply Lyapunov and invariance principles. In particular, the pair \( (V, f -\clco A_0 - N_C) \) is decreasing if and only if the epigraph \( \operatorname{epi} V \) is invariant under the extended dynamic
\[
(\dot{x}(t), \dot{y}(t)) \in (f(x(t)) -\clco A_0(x(t)) - N_C(x(t)),\ 0).
\]
When \( V \) is the indicator function of a set \( S \), this corresponds to the classical notion of set invariance.

The use of the regularization based on \( F_k \) facilitates the analysis of Lyapunov stability and invariance, while ensuring the preservation of the essential structural features of the original dynamics.

Throughout this section, the regularized systems are introduced solely as auxiliary analytical tools. Their role is to establish Hamiltonian and Lyapunov inequalities that can be transferred directly to the original differential inclusion, from which the stability properties are deduced.

The next lemma, as stated in Lemma \ref{Thibault},  is a direct consequence of \cite[Theorem~3.1]{Thibault2013} when applied to the case where the set $C$ is convex instead of $r$-prox-regular. It provides useful  information about the velocity $\dot{x}(t)$ of the solution of \eqref{eq:prob}. The proof follows the results in \cite{Thibault2013}, with some differences specific to our case that we will highlight in the key steps.
\begin{lemma}
\label{Thibault}
Suppose that 
\eqref{eq:uniform-boundedness} holds and there exist \( \rho, M_f > 0 \) such that, for all \( x \in \B (x_0, \rho) \),
\begin{equation}
\label{bound}
\|f(x)\| \leq M_f.    
\end{equation}
Let $x(t)$ be the solution of \eqref{eq:prob}. Then there exist functions \( g(x(t)) \in \clco A_0(x(t)) \) and \( \eta(t) \in N_C(x(t)) \) such that
\[
\dot{x}(t) = f(x(t)) - g(x(t)) - \eta(t) \text{~~and~~} \|\eta(t)\| \leq \|f(x(t)) - g(x(t))\|
\]
for almost every $t \geq 0$. Moreover, the solution \( x(t) \) is Lipschitz continuous.
\end{lemma}
\begin{proof}
The argument proceeds in three main steps. First, we introduce a regularized dynamics based on a fixed Lipschitz approximation of $\clco A_0$, which ensures existence, uniqueness, and uniform Lipschitz bounds for the associated trajectories. Second, we establish compactness and pass to the limit in the regularized system, relying on uniform estimates and the outer semicontinuity of $\clco A_0$ and $N_C$, to recover a decomposition of the velocity of the original solution. Finally, we derive the key estimate on the normal component, which yields Lipschitz regularity of the solution and the stated velocity bound.

We aim to analyze the structure of the velocity \( \dot{x}(t) \) of the solution of problem \eqref{eq:prob}, given by
\[ 
\dot{x}(t) \in f(x(t)) - \clco A_0(x(t)) - N_C(x(t)),
\]
by studying a regularized system associated with the Lipschitz outer
approximation provided by Theorem~\ref{smirnov}. The corresponding single-valued selection
\[
\psi_k(x) := s\bigl(F_k(x)\bigr),
\]
defined via the Steiner point and introduced in
Section~\ref{sec:selection}, enjoys key regularity properties: it is
uniformly bounded by \( b \) and locally Lipschitz continuous on
\( \R^n \).

Fix \( \varepsilon > 0 \). We choose a corresponding index \( k \) (fixed throughout) such that, for all $x \in \R^n$,
\[ 
F_k(x) \subseteq \clco A_0(x) + \varepsilon \B,
\]
where each \( F_k(x) \) is nonempty, compact, convex, and uniformly bounded by \( b \B \).

Let \( T > 0 \) be arbitrary. Consider the regularized system
\begin{align}\label{eq:prob_lambda} 
\left\{\begin{array}{ll}\dot{x}_\lambda(t) &= f(x_\lambda(t)) - \psi_k(x_\lambda(t)) - \frac{1}{2\lambda} \nabla \mathbf{d}^2(x_\lambda(t); C), \\
 x_\lambda(0) &= x_0 \in C.
\end{array}\right.
\tag{\( \text{P}_\lambda \)}
\end{align}
Since all terms on the right-hand side are locally Lipschitz, this system admits a unique absolutely continuous solution \( x_\lambda(\cdot) \) on \( [0, T] \). It is well known that 
\begin{itemize}
    \item $\mathbf{d}(x;C) = \Vert x - \proj_C (x) \Vert $
    \item $\nabla \mathbf{d}^2(x;C) = 2\big( x - \proj_C (x)\big);$ and 
  \item the mapping $\proj_C (\cdot)$ is Lipschitz continuous of rank $1$. 
\end{itemize}
Now, for $t \in [0,T]$, from \eqref{eq:prob_lambda} we have 
\[ 
\dot{x}_\lambda(t) = f(x_\lambda(t)) - \psi_k(x_\lambda(t)) - \frac{1}{\lambda} (x_\lambda(t) - \proj_C(x_\lambda(t))).
\]
Hence,
\[ 
\left\| \dot{x}_\lambda(t) - f(x_\lambda(t)) + \psi_k(x_\lambda(t)) \right\| = \displaystyle{\frac{1}{\lambda} \Vert x_\lambda (t) - \proj_C (x_\lambda (t))\Vert} = \frac{1}{\lambda} \mathbf{d}(x_\lambda(t); C).
\]
Define \( \beta := M_f + b \). Since both \( f \) and \( \psi_k \) are uniformly bounded on bounded subsets, we apply an estimate of the same type as in \cite[Lemma 3.1]{Thibault2013} to get
\[ 
\left\| \dot{x}_\lambda(t) - f(x_\lambda(t)) + \psi_k(x_\lambda(t)) \right\| \leq \beta(1 - e^{-t/\lambda}) \leq \beta.
\]
Consequently, we have
\[ 
\|\dot{x}_\lambda(t)\|  \leq 2\beta.
\]
Therefore,  for each $\lambda >0$, the family \( (x_\lambda) \) is Lipschitz continuous on \( [0, T] \), with a Lipschitz constant independent of \( \lambda \). Taking  a sequence \( (\lambda_n) \), with \( \lambda_n \downarrow 0 \), then the   corresponding sequence \( (x_{\lambda_n}) \) is uniformly Lipschitz  and by the Arzel\`{a}--Ascoli theorem, it admits a uniformly convergent subsequence (do not relable) on \( [0, T] \).  We denote its limit by \( x(\cdot) \), so that \( x_{\lambda_n} \to x \) uniformly on \( [0,T] \), with \( x \in C([0,T]; \R^n) \).

We now examine the limit behavior of the velocities. Since \( F_k(x_{\lambda_n}(t)) \subseteq \clco A_0(x_{\lambda_n}(t)) + \varepsilon \B \),  
for each \( n \in \mathbb{N} \), there exists a point \( \tilde{g}_{\lambda_n}(t) \in \clco A_0(x_{\lambda_n}(t)) \) such that
\[
\|\psi_k(x_{\lambda_n}(t)) - \tilde{g}_{\lambda_n}(t)\| \leq \varepsilon.
\]
The sequences \( (f(x_{\lambda_n}(t))) \), \( (\psi_k(x_{\lambda_n}(t))) \), and \( (\tilde{g}_{\lambda_n}(t)) \) are uniformly bounded,  so by extracting a subsequence if necessary, we may assume that \( \tilde{g}_{\lambda_n}(t) \to g(x(t)) \in \R^n \) almost everywhere. By the outer semicontinuity of \( \clco A_0 \), it follows that
\[
g(x(t)) \in \clco A_0(x(t)).
\]
Next, define
\[
\eta_{\lambda_n}(t) := \frac{1}{\lambda_n} (x_{\lambda_n}(t) - \proj_C(x_{\lambda_n}(t))) \in N_C(x_{\lambda_n}(t)),
\]
so that
\[
\dot{x}_{\lambda_n}(t) = f(x_{\lambda_n}(t)) - \psi_k(x_{\lambda_n}(t)) - \eta_{\lambda_n}(t).
\]
Passing to the limit and using the closedness of the graph of \( N_C \), we obtain
\[
\dot{x}(t) = f(x(t)) - g(x(t)) - \eta(t),
\]
with \( \eta(t) \in N_C(x(t)) \) and \( g(x(t)) \in \clco A_0(x(t)) \).

To estimate \( \|\eta(t)\| \), define
\[
\Phi(t) := f(x(t)), \quad \vartheta(t) := g(x(t)).
\]
Since \( -\dfrac{\eta(t)}{\|\eta(t)\|} \in N_C(x(t)) \), for all \( s < t \) and \( x(s) \in C \),
\[
\left\langle -\frac{\eta(t)}{\|\eta(t)\|}, x(s) - x(t) \right\rangle \leq 0.
\]
Defining \( \Omega(t) := x(t) - \Phi(t) + \vartheta(t) \), and applying the argument from \cite[Theorem 3.1]{Thibault2013}, we deduce
\[
\left\langle \frac{\eta(t)}{\|\eta(t)\|}, \frac{\Omega(t) - \Omega(s)}{t - s} \right\rangle
\leq \left\langle -\frac{\eta(t)}{\|\eta(t)\|}, \frac{\Phi(t) - \Phi(s)}{t - s} - \frac{\vartheta(t) - \vartheta(s)}{t - s} \right\rangle.
\]
Letting \( s \to t \), this yields
\[
\|\eta(t)\| \leq \|f(x(t)) - g(x(t))\|.
\]
Finally, using the bounds \( \|f(x(t))\| \leq M_f \) and \( \|g(x(t))\| \leq b \), we obtain
\[
\|\dot{x}(t)\| \leq M_f + b.
\]
Thus, the limit trajectory \( x(\cdot) \) is Lipschitz continuous on \( [0, T] \) with Lipschitz constant at most \( M_f + b \).
\end{proof}
The upcoming theorem holds significant importance. It is worth highlighting that the methodology employed to establish Theorem \ref{t:invariance} differs from the approaches outlined in the references \cite{clarke2}  and \cite{adly2018invariant}.  Our proof, in essence, relies on a geometric approach, including elements of proximal analysis and based on the properties of the dynamics \eqref{eq:prob}. We define the corresponding \emph{lower Hamiltonian} to the dynamic \eqref{eq:prob}  by 
    \begin{equation}
        \label{eq:hamiltonian}
              \boldsymbol{h}\big(x(t),\zeta):= \inf_{g(x(t)) \in \clco A_0(x(t))}\,\,\inf_{\eta(t) \in N_C (x(t))} \big\langle \zeta, f(x(t)) - g(x(t)) - \eta(t) \big\rangle.
    \end{equation}
For simplicity, we use $g$ and $\eta$ instead of $g(x(t))$ and $\eta(t)$.

To handle the nonsmooth nature of the original dynamics \eqref{eq:prob} and facilitate the establishment of our main result, the analysis focuses on controlling the lower Hamiltonian associated with \( \clco A_0 \), rather than on approximating individual trajectories. This control is achieved through the Lipschitz outer approximations \( F_k \) constructed in Theorem~\ref{smirnov}. The following proposition formalizes this Hamiltonian approximation property and constitutes the key mechanism by which stability estimates obtained for the regularized systems are transferred to the original differential inclusion.
\begin{remark} \label{rem:Hamilton}
In this work, the term \emph{Hamiltonian} refers to a Lyapunov-like function used to analyze the energy dissipation or decrease along trajectories. Unlike the classical Hamiltonian in conservative systems (which is typically conserved), here it plays a variational role and may strictly decrease. This terminology is used in a generalized sense, consistent with frameworks involving nonsmooth or monotone differential inclusions.
\end{remark}
\begin{proposition}[Approximation of the Hamiltonian]
\label{prop:hamiltonian-approx}
Let \((F_k)_{k\in\mathbb{N}}\) be the Lipschitz approximations of
\(\clco A_0\) provided by Theorem~\ref{smirnov}. For each \(k\), define the approximate lower Hamiltonian by
\[
\boldsymbol{h}_k(x,\zeta)
:= \min_{g \in F_k(x)} \inf_{\eta \in N_C(x)}
\left\langle \zeta, f(x) - g - \eta \right\rangle .
\]
Let \(\boldsymbol{h}\) denote the lower Hamiltonian associated with \(\clco A_0\),
as defined in \eqref{eq:hamiltonian}. Then, for every \(x \in C \), every \(\zeta \in \R^n\), and every \(\varepsilon>0\), there exists an integer \(k(\varepsilon,x)\) such that, for all \(k \ge k(\varepsilon,x)\),
\begin{equation}
\label{approx}
\boldsymbol{h}(x,\zeta)
\le \boldsymbol{h}_k(x,\zeta) + \varepsilon \|\zeta\|.
\end{equation}
\end{proposition}
\begin{proof}
Fix \( x \in C \) and \( \zeta \in \R^n \). Let \( g_k \in F_k(x) \) such that 
\[
g_k \in \argmin_{g \in F_k(x)} \inf_{\eta \in N_C(x)} \left\langle \zeta, f(x) - g - \eta \right\rangle.
\]
Thus,
\[
\boldsymbol{h}_k(x,\zeta) = \inf_{\eta \in N_C(x)} \left\langle \zeta, f(x) - g_k - \eta \right\rangle.
\]
Since \( F_k(x) \subseteq \clco A_0(x) + \varepsilon \B \), there exists \( \tilde{g} \in \clco A_0(x) \) such that
\[
\|g_k - \tilde{g}\| \leq \varepsilon.
\]
Then for every \( \eta \in N_C(x) \),
\[
\left\langle \zeta, f(x) - \tilde{g} - \eta \right\rangle = \left\langle \zeta, f(x) - g_k - \eta \right\rangle + \left\langle \zeta, g_k - \tilde{g} \right\rangle,
\]
and by the Cauchy--Schwarz inequality,
\[
|\langle \zeta, g_k - \tilde{g} \rangle| \leq \|\zeta\| \|g_k - \tilde{g}\| \leq \varepsilon \|\zeta\|.
\]
Thus, for every \( \eta \in N_C(x) \),
\[
\left\langle \zeta, f(x) - \tilde{g} - \eta \right\rangle \leq \left\langle \zeta, f(x) - g_k - \eta \right\rangle + \varepsilon \|\zeta\|.
\]
Taking the infimum over \( \eta \in N_C(x) \), we get
\[
\inf_{\eta \in N_C(x)} \left\langle \zeta, f(x) - \tilde{g} - \eta \right\rangle \leq \boldsymbol{h}_k(x,\zeta) + \varepsilon \|\zeta\|.
\]
Since \( \tilde{g} \in \clco A_0(x) \), it follows that
\[
\boldsymbol{h}\big(x,\zeta) \leq \inf_{\eta \in N_C(x)} \left\langle \zeta, f(x) - \tilde{g} - \eta \right\rangle,
\]
thus,
\[
\boldsymbol{h}\big(x,\zeta) \leq \boldsymbol{h}_k(x,\zeta) + \varepsilon \|\zeta\|,
\]
which proves the claim.
\end{proof}
\section{Proofs of the Main Results}
\label{sec:proofs-main-results}
In this section, we provide the proofs of the results stated in Section~\ref{sec:main-results}. The arguments rely on the Hamiltonian framework and approximation tools developed in Section~\ref{sec:stability-framework}. We also recall a technical variational result from~\cite{radulescu1997geometric}, which is used in the proof of Theorem~\ref{t:invariance} to handle horizontal proximal normals to $\epi V$.
\begin{theorem}[{\cite[Theorem~2.4]{radulescu1997geometric}}]
\label{Clarke+Radu}
Let $V: \R^n \to \Rb$ be a lower semicontinuous extended-real-valued function, $x\in \dom V$ and $(\nu^*,0) \in N_{\epi V}^P (x,V(x))$ with $\nu^* \neq 0$. Then, for any $\vp > 0$, there exists $\bar{x} \in \ball (y,\vp) \cap \dom V$ with $\vert V(\bar{x})- V(x)\vert < \vp, \mu \in (0,\vp),$ and $\xi \in \ball (\nu^*,\vp)$ such that $(\xi,-\mu) \in  N_{\epi V}^P (\bar{x},V(\bar{x})).$
\end{theorem}
\subsection{Proof of Theorem~\ref{t:invariance}}
We will prove that $\big(\epi V, (f-\clco A_0-N_C)\times \{0\}\big)$ is invariant in the following two steps.

\textbf{Step 1:} Assumption~\eqref{eq:H1} is equivalent to say that, for all $\zeta \in \partial_P V(x)$, 
\[
\displaystyle\inf_{g \in \clco A_0(x)}\inf_{\eta \in N_C (x)} \big\langle \zeta, f(x) - g - \eta \big\rangle \leq 0.
\]
The goal is to prove that for all $(\xi,\mu) \in N^P_{\epi V} (x,\alpha)$ there exists $g \in \clco A_0(x)$ and $\eta \in N_C (x)$ such that $\big\langle \xi, f(x)-g - \eta\big\rangle \leq 0.$

First, we will demonstrate that $\mu \leq 0.$ Let $z \in \dom V$ and let $(\nu^*,0) \in N^P_{\epi V} (z,V(z))$ with $\nu^* \neq 0$. Without loss of generality, we assume that $\Vert \nu^* \Vert = 1$. Since $(\nu^*,0) \in N^P_{\epi V} (z,V(z))$ then there exists a point 
    $(x,V(z)) \notin \epi V$ where $(z,V(z))$ is the closest point in $\epi V$ to $(x,V(z))$ i.e. for $s \in [0,1]$,
\begin{align*}
&\big(z,V(z)\big) \in \proj_{\epi V} \big(x,V(z)\big) \\
\iff & \big(z,V(z)\big) \in \proj_{\epi V} \big((z,V(z)) +s(x-z,0)\big)\\
\iff & \big(z,V(z)\big) \in \proj_{\epi V} \big(z +s(x-z),V(z)\big).
\end{align*}
Then, for all $s \in [0,1]$,
\[
(\nu^*,0) \in \partial_P \mathbf{d}\big((z +s(x-z),V(z));\epi V \big). 
\]
Thus, according to Proposition \ref{proxdist}, we can deduce that 
\[
(\nu^*,0) = \left\{ \left(\frac{x-z}{\Vert x-z \Vert},0\right)\right\} \implies \nu^* = \frac{x-z}{\Vert x-z \Vert}.
\]
Now, since $V$ is lower semicontinuous and by the definition of the epigraph of $V$, we have $\mathbf{d}(x,\alpha);\epi V)$ is a nonincreasing function as a function of $\alpha$. Thus, for all $(\bar{x}, V(\bar{z}))$ and all $\vp > 0$,
\begin{align}
\label{eqdist}
    \mathbf{d} \big((\bar{x}, V(\bar{z}));\epi V\big) \leq \mathbf{d}\big((\bar{x}, V(\bar{z})-\vp);\epi V\big). 
\end{align}
Suppose now that the point $(\bar{x}, V(\bar{z}))$ is close to the point $(x,V(z))$, then by \eqref{eqdist} we distinguish the following two different cases. 

\emph{Case 1.1:} $\mathbf{d}\big((\bar{x}, V(\bar{z}));\epi V\big) < \mathbf{d}\big((\bar{x}, V(\bar{z})-\vp); \epi V\big)$.
     Since $(\bar{x}, V(\bar{z}))$ is close to the point $(x,V(z))$ then, according to \cite[Theorem 1.4]{radulescu1997geometric}, there exists $(x,\mu) \in \partial_P \mathbf{d}\big((\bar{x},V(\bar{z}));\epi V\big)$ such that 
     \begin{align*}
     & \big\langle (x,\mu),(\bar{x},V(\bar{z})-\vp) - (\bar{x},V(\bar{z})))\big\rangle > 0 \\  
     \implies & \big\langle (x,\mu),(0,-\vp)\big\rangle > 0\\
     \implies & -\mu \vp > 0\\
     \implies & \mu < 0 \,\,\,\textrm{ (since $\vp > 0$) }. 
     \end{align*}

\emph{Case 1.2:} $\mathbf{d}\big((\bar{x}, V(\bar{z}));\epi V\big) = \mathbf{d}\big((\bar{x}, V(\bar{z})-\vp); \epi V\big)$.
     We have 
     \begin{align*}
     & \langle (x,\mu),(\bar{x},V(\bar{z})-\vp) - (\bar{x},V(\bar{z})))\rangle = 0 \\  
     \implies & \langle (x,\mu),(0,-\vp)\rangle > 0\\
     \implies & \mu \vp = 0\\
     \implies & \mu = 0 \,\,\,\textrm{ (since $\vp > 0$) }. 
     \end{align*}
Thus, we deduce that $\mu \leq 0$.

\textbf{Step 2:} Here we have two cases: 

\emph{Case 2.1:} \( \mu < 0.\)
Since \( (\xi, \mu) \in N^P_{\epi V}(x, \alpha) \) with \( \mu < 0 \), it follows that
\[
-\frac{\xi}{\mu} \in \partial_P V(x).
\]
By assumption~\eqref{eq:H1}, we have
\[
\boldsymbol{h}\left(x, -\frac{\xi}{\mu}\right) \leq 0.
\]
Since the infimum defining \( \boldsymbol{h}\big(x,\zeta) \) may not be attained, we proceed by approximation.

Fix an arbitrary \( \varepsilon > 0 \), and consider the Lipschitz outer approximation \( F_k \) of \( \clco A_0 \) constructed in
Theorem~\ref{smirnov}, corresponding to \( \varepsilon \), satisfying
\[
F_k(x) \subseteq \clco A_0(x) + \varepsilon \B,
\qquad \text{for all } x \in \R^n.
\]
Moreover, the associated Steiner selection \( \psi_k \) is defined by
\[
\psi_k(x) := s\bigl(F_k(x)\bigr),
\]
and satisfies \( \psi_k(x) \in F_k(x) \) for all \( x \in \R^n \).
By the Hamiltonian approximation property~\eqref{approx}, we know that
\[
\boldsymbol{h}\left(x, -\frac{\xi}{\mu}\right) \leq \boldsymbol{h}_k\left(x, -\frac{\xi}{\mu}\right) + \varepsilon \left\| \frac{\xi}{\mu} \right\|,
\]
where
\[
\boldsymbol{h}_k(x,\zeta) := \min_{g \in F_k(x)} \inf_{\eta \in N_C(x)} \langle \zeta, f(x) - g - \eta \rangle.
\]
Since \( \psi_k(x) \in F_k(x) \), and by definition of \( \boldsymbol{h}_k \), it follows that
\[
\inf_{\eta \in N_C(x)} \left\langle -\frac{\xi}{\mu}, f(x) - \psi_k(x) - \eta \right\rangle \geq \boldsymbol{h}_k\left(x, -\frac{\xi}{\mu}\right).
\]
Moreover, since \( \boldsymbol{h}_k\left(x, -\dfrac{\xi}{\mu}\right) \geq -\varepsilon \left\| \dfrac{\xi}{\mu} \right\| \),
we can use standard properties of convex analysis: since \( N_C(x) \) is closed and convex, and the mapping
\[
\eta \mapsto \left\langle -\frac{\xi}{\mu}, f(x) - \psi_k(x) - \eta \right\rangle
\]
is affine (and thus continuous) in \( \eta \), it follows that for every \( \delta > 0 \), there exists \( \eta \in N_C(x) \) such that
\[
\left\langle -\frac{\xi}{\mu}, f(x) - \psi_k(x) - \eta \right\rangle \leq \boldsymbol{h}_k\left(x, -\frac{\xi}{\mu}\right) + \delta.
\]
Applying this with \( \delta = \varepsilon \left\| \frac{\xi}{\mu} \right\| \) yields the existence of \( \eta \in N_C(x) \) such that
\[
\left\langle -\frac{\xi}{\mu}, f(x) - \psi_k(x) - \eta \right\rangle \leq \varepsilon \left\| \frac{\xi}{\mu} \right\|.
\]
Multiplying both sides by \( -\mu > 0 \), we deduce
\[
\left\langle \xi, f(x) - \psi_k(x) - \eta \right\rangle \leq \varepsilon \|\xi\|.
\]
Now, since \( \psi_k(x) \in F_k(x) \subseteq \clco A_0(x) + \varepsilon \B \),  
there exists \( \tilde{g} \in \clco A_0(x) \) such that
\[
\|\psi_k(x) - \tilde{g}\| \leq \varepsilon.
\]
Thus, we can estimate
\[
\begin{aligned}
\left\langle \xi, f(x) - \tilde{g} - \eta \right\rangle
&= \left\langle \xi, f(x) - \psi_k(x) - \eta \right\rangle + \left\langle \xi, \psi_k(x) - \tilde{g} \right\rangle \\
&\leq \varepsilon \|\xi\| + \varepsilon \|\xi\| \\
&= 2\varepsilon \|\xi\|.
\end{aligned}
\]
Since \( \varepsilon > 0 \) was arbitrary, sending \( \varepsilon \to 0 \) yields the existence of \( \tilde{g} \in \clco A_0(x) \) and \( \eta \in N_C(x) \) such that
\[
\left\langle \xi, f(x) - \tilde{g} - \eta \right\rangle \leq 0.
\]
Thus, we conclude the proof.

\emph{Case 2.2:} \( \mu = 0. \)
Since \( (\xi, 0) \in N^P_{\epi V}(x, V(x)) \), by Theorem~\ref{Clarke+Radu}, there exist sequences \( (x_i)_i \subset C \), \( (\xi_i)_i \subset \R^n \), and \( (\theta_i)_i \subset (0,+\infty) \) such that
\[
(\xi_i, -\theta_i) \in N^P_{\epi V}(x_i, V(x_i)),
\quad
\xi_i \to \xi,
\quad
\theta_i \to 0,
\quad
x_i \to x,
\quad
V(x_i) \to V(x).
\] 
For each \( i \), we define
$\zeta_i := -\displaystyle\frac{\xi_i}{\theta_i}$. Then \( \zeta_i \in \partial_P V(x_i) \). Applying assumption~\eqref{eq:H1} at each point \( x_i \), we have
\[
\boldsymbol{h}\big(x_i,\zeta_i) \leq 0,
\]
where the lower Hamiltonian associated with \eqref{eq:prob} is given by
\[
\boldsymbol{h}\big(x_i,\zeta_i) := \inf_{g \in \clco A_0(x_i)} \inf_{\eta \in N_C(x_i)} \left\langle \zeta_i, f(x_i) - g - \eta \right\rangle.
\]
Fix an arbitrary \( \varepsilon > 0 \).
Using the Lipschitz approximation \( F_k \) of \( \clco A_0 \) (Theorem~\ref{smirnov}), we have
\[
\boldsymbol{h}\big(x_i,\zeta_i) \leq \boldsymbol{h}_k(x_i,\zeta_i) + \varepsilon \|\zeta_i\|,
\]
where
\[
\boldsymbol{h}_k(x_i,\zeta_i) := \min_{g \in F_k(x_i)} \inf_{\eta \in N_C(x_i)} \left\langle \zeta_i, f(x_i) - g - \eta \right\rangle.
\]
Thus, for each \( i \),
\[
\boldsymbol{h}_k(x_i,\zeta_i) \geq -\varepsilon \|\zeta_i\|.
\]
By the definition of \( \boldsymbol{h}_k \), there exist points \( g_i \in F_k(x_i) \) and \( \eta_i \in N_C(x_i) \) such that
\[
\left\langle \zeta_i, f(x_i) - g_i - \eta_i \right\rangle \leq \varepsilon \|\zeta_i\|.
\]
Recalling that \( \zeta_i = -\displaystyle\frac{\xi_i}{\theta_i} \),  
multiplying both sides by \( \theta_i > 0 \) yields
\[
\left\langle -\xi_i, f(x_i) - g_i - \eta_i \right\rangle \leq \varepsilon \|\xi_i\|,
\]
or equivalently,
\[
\left\langle \xi_i, f(x_i) - g_i - \eta_i \right\rangle \geq -\varepsilon \|\xi_i\|.
\]
Since \( g_i \in F_k(x_i) \subseteq \clco A_0(x_i) + \varepsilon \B \), 
there exists \( \tilde{g}_i \in \clco A_0(x_i) \) such that
\[
\|g_i - \tilde{g}_i\| \leq \varepsilon.
\]
Thus, 
\begin{align*}
\left\langle \xi_i, f(x_i) - \tilde{g}_i - \eta_i \right\rangle
&= \left\langle \xi_i, f(x_i) - g_i - \eta_i \right\rangle + \left\langle \xi_i, g_i - \tilde{g}_i \right\rangle \\
&\geq -\|\xi_i\| \|f(x_i) - g_i - \eta_i\| - \|\xi_i\| \|g_i - \tilde{g}_i\|\\
&\geq -\varepsilon \|\xi_i\| - \varepsilon \|\xi_i\|  \\
&= -2\varepsilon \|\xi_i\|.
\end{align*}
Moreover, for \( \beta := M_f + b \), using the estimation from Lemma~\ref{Thibault}, we can deduce that at \( x_i \),
\[
\begin{aligned}
\|\eta_i\| 
&\leq \| f(x_i) - \tilde{g}_i \| \\
&\leq \|f(x_i)\| + \|\tilde{g}_i\| \\
&\leq M_f + b = \beta.
\end{aligned}
\]
Therefore, considering the boundedness of \( f(x_i) \), \( \tilde{g}_i \), and \( \eta_i \), 
along with the Lipschitz continuity of \( f \), the local boundedness of \( \clco A_0 \),
and the closedness of \( N_C \),
we can extract subsequences (without relabeling) such that
\[
\tilde{g}_i \to \tilde{g} \in \clco A_0(x) 
\quad \text{and} \quad 
\eta_i \to \eta \in N_C(x).
\]
Passing to the limit as \( i \to +\infty \) in the inequality
\[
\left\langle \xi_i, f(x_i) - \tilde{g}_i - \eta_i \right\rangle \geq -2\varepsilon \|\xi_i\|,
\]
and using the convergences \( x_i \to x \), \( \xi_i \to \xi \), \( \tilde{g}_i \to \tilde{g} \in \clco A_0(x) \), and \( \eta_i \to \eta \in N_C(x) \),
we obtain
\[
\left\langle \xi, f(x) - \tilde{g} - \eta \right\rangle \geq -2\varepsilon \|\xi\|.
\]
Since \( \varepsilon > 0 \) was arbitrary, letting \( \varepsilon \to 0 \) yields
\[
\left\langle \xi, f(x) - \tilde{g} - \eta \right\rangle \geq 0.
\]
Collecting the results of Steps 1 and 2, we deduce that, for every \( (\xi,\mu) \in N^P_{\epi V}(x,V(x)) \),
there exist \( \tilde{g} \in \clco A_0(x) \) and \( \eta \in N_C(x) \) such that
\[
\begin{cases}
\big\langle \xi, f(x) - \tilde{g} - \eta \big\rangle \leq 0 & \text{if } \mu < 0, \\
\big\langle \xi, f(x) - \tilde{g} - \eta \big\rangle \geq 0 & \text{if } \mu = 0.
\end{cases}
\]
In both cases, the invariance condition is satisfied.  
Therefore, \( \big(\epi V, (f - \clco A_0 - N_C) \times \{0\}\big) \) is invariant, and the pair \( (V, f - \clco A_0 - N_C) \) is decreasing.

\ref{t:invariance_main}: Since the pair $\big(V,f-\clco A_0-N_C\big)$ is decreasing, it means that, for every $x_0 \in C$ the trajectory $x(\cdot)$ of \eqref{eq:prob} defined on $[0,\infty]$ and starting from $x(0) = x_0$, we have $V(x(t)) \leq V(x_0)$ for all $t \geq 0$. Now, take $\alpha = V(x_0)$, we can show that $\big([V\leq\alpha]_{\mid{\dom V}} ,f-\clco A_0-N_C\big)$ is invariant. 
\begin{remark}[Geometric interpretation]
\label{rem:horizontal_normals}
The two cases \( \mu < 0 \) and \( \mu = 0 \) correspond to different geometric behaviors relative to the epigraph of \( V \).
\begin{enumerate}
    \item When \( \mu < 0 \), the pair \( (\xi,\mu) \in N^P_{\epi V}(x,V(x)) \) corresponds to a \emph{strict proximal normal}, pointing outward from the epigraph. In this case, the condition
    \[
    \left\langle \xi, f(x) - \tilde{g} - \eta \right\rangle \leq 0
    \]
    ensures that the dynamics is directed inward, or at least non-expanding, relative to the boundary of the epigraph. This guarantees strong invariance.
    
    \item When \( \mu = 0 \), the pair \( (\xi,0) \in N^P_{\epi V}(x,V(x)) \) corresponds to a \emph{horizontal normal}. In this situation, strict inward movement is not required: it is sufficient that the dynamics do not push strictly outward across the boundary. Thus, the inequality
    \[
    \left\langle \xi, f(x) - \tilde{g} - \eta \right\rangle \geq 0
    \]
is fully consistent with the invariance of the epigraph under the dynamics.
\end{enumerate}
\end{remark}
\subsection{Proof of Theorem~\ref{ConvW}}
The proof of the first part closely follows the ideas developed in \cite{clarke2}, adapted to our setting.

Let $(V,W)$ be a Lyapunov pair for the system~\eqref{eq:prob}, satisfying assumption~\eqref{eq:H2}. 
Following a standard regularization technique, we define, for each \( n > 0 \), the infimal convolution
\[
W_n(x) := \inf_{y \in \R^n} \left\{ W(y) + n \|x-y\|^2 \right\}.
\]
It is well known that each \( W_n \) is Lipschitz continuous and that \( W_n(x) \to W(x) \) pointwise as \( n \to +\infty \) (see, e.g., \cite{clarke2}). Moreover, for each \( n \) and \( x \in \R^n \),
\[
W_n(x) > 0 \quad \text{if and only if} \quad W(x) > 0,
\]
thus preserving the positivity structure of \( W \). Furthermore, it is clear that $(V,W_n)$ is a Lyapunov pair for the system \eqref{eq:prob}.

Next, define the augmented system on \( \R^n \times \R \) by
\[
\mathcal{A}(x,y_n) := \left( f(x) - \clco A_0(x) - N_C(x) \right) \times \left\{ W_n(x) \right\}.
\]
We also define the augmented Lyapunov function
\[
\mathcal{V}_n(x,y_n) := V(x) + y_n.
\]
We claim that the pair \( (\mathcal{V}_n, \mathcal{A}) \) is decreasing. Indeed, let \( (x,y_n) \in \R^n \times \R \), and let \( (\zeta, \theta) \in \partial_P \mathcal{V}_n(x,y_n) \). Then, by the sum rule for proximal subdifferentials, we have
\[
\zeta \in \partial_P V(x), \quad \theta = 1.
\]
By assumption~\eqref{eq:H2}, it follows that
\[
\boldsymbol{h}\left(x, \zeta\right) \leq -W(x).
\]
Applying Proposition~\ref{prop:hamiltonian-approx}, we obtain
\[
\boldsymbol{h}\left(x, \zeta\right) \leq \boldsymbol{h}_k(x,\zeta) + \varepsilon \|\zeta\|.
\]
Since \( \boldsymbol{h}_k(x,\zeta) \) is defined via
\[
\boldsymbol{h}_k(x,\zeta) := \min_{g \in F_k(x)} \inf_{\eta \in N_C(x)} \left\langle \zeta, f(x) - g - \eta \right\rangle,
\]
and \( \psi_k(x) \in F_k(x) \), we deduce that there exists \( \eta \in N_C(x) \) such that
\[
\left\langle \zeta, f(x) - \psi_k(x) - \eta \right\rangle \leq \boldsymbol{h}_k(x,\zeta) + \varepsilon \|\zeta\| \leq -W_n(x) + \varepsilon \|\zeta\|.
\]
Moreover, since \( \psi_k(x) \in F_k(x) \subseteq \clco A_0(x) + \varepsilon \B \), there exists \( \tilde{g} \in \clco A_0(x) \) satisfying
\[
\|\psi_k(x) - \tilde{g}\| \leq \varepsilon,
\]
and hence
\begin{align*}
\left\langle \zeta, f(x) - \tilde{g} - \eta \right\rangle
&= \left\langle \zeta, f(x) - \psi_k(x) - \eta \right\rangle + \left\langle \zeta, \psi_k(x) - \tilde{g} \right\rangle \\
&\leq \left( -W_n(x) + \varepsilon \|\zeta\| \right) + \varepsilon \|\zeta\| \\
&= -W_n(x) + 2\varepsilon \|\zeta\|.
\end{align*}
Thus, the Hamiltonian associated to the augmented dynamics satisfies:
\[
\left\langle (\zeta,1), (f(x)-\tilde{g}-\eta, W_n(x)) \right\rangle \leq 2\varepsilon \|\zeta\|.
\]
Letting \( \varepsilon \to 0 \), we deduce the infinitesimal decrease condition for \( (\mathcal{V}_n, \mathcal{A}) \).

Applying the same argument used in the proof of Theorem \ref{t:invariance}, we deduce the existence of a Lipschitz continuous solution \( (x(\cdot),y_n(\cdot)) \) of the system
\[
\begin{cases}
\dot{x}(t) \in f(x(t)) - \clco A_0(x(t)) - N_C(x(t)),\\
\dot{y}_n(t) = W_n(x(t)),
\end{cases}
\]
with initial condition \( (x(0),y_n(0)) = (x_0,0) \), such that
\[
\mathcal{V}_n(x(t), y_n(t)) \leq \mathcal{V}_n(x_0,0) = V(x_0).
\]
Unfolding the definition of \( \mathcal{V}_n \), we obtain that, for all $t \geq 0$,
\[
V(x(t)) + y_n(t) \leq V(x_0).
\]
Since
\[
y_n(t) = \int_0^t W_n(x(s))\, ds,
\]
it follows that
\[
V(x(t)) + \int_0^t W_n(x(s))\, ds \leq V(x_0).
\]
Since \( V \) is bounded below and \( V(x(t)) \geq \inf V \) for all \( t \), we deduce that \( \displaystyle\int_0^{+\infty} W_n(x(s))\, ds \) is finite. In particular, \( W_n(x(t)) \to 0 \) as \( t \to +\infty \).

Finally, since \( W_n \to W \) pointwise and uniformly on compact sets, we conclude that \( W(x(t)) \to 0 \) as \( t \to +\infty \). Equivalently,
\[
\lim_{t\to+\infty} \mathbf{d}(x(t); W^{-1}(0)) = 0.
\]
If additionally each point of \( W^{-1}(0) \) is Lyapunov stable, standard arguments imply that \( W^{-1}(0) \) is $\mathbf{PAS}$.
\subsection{Proofs of Proposition~\ref{ConvE} and Theorem~\ref{SS}}
\emph{Proof of Proposition~\ref{ConvE}.}
The proof is straightforward, relying on the fact that $\mathcal{E}$ is an
invariant set containing the $\omega$-limit set and each point of the set is
Lyapunov stable.
\hfill$\square$

\medskip
\noindent
\emph{Proof of Theorem~\ref{SS}.}
From Theorem \ref{ConvW}, we can deduce the existence of a solution of \eqref{eq:prob} that converges to $W^{-1}(0)$. Now, since every point in $W^{-1}(0)$ is Lyapunov stable, we can deduce that $W^{-1}(0) \subseteq \mathcal{E}$ and thus $\mathcal{E} = W^{-1}(0)$. Using Proposition \ref{ConvE}, we can easily prove that $W^{-1}(0)$ is $\mathbf{SS}$.
\section{Applications}
\label{application}
We now illustrate the applicability of the developed stability framework through two types of dynamical systems: smooth inertial systems involving second-order dynamics with Hessian-driven damping, and nonsmooth differential inclusions arising in unilateral mechanics. In each case, we verify that the assumptions of the main stability results are satisfied, and we characterize the asymptotic behavior of trajectories.
\begin{example}[Smooth inertial system with Hessian damping]
\label{exattouch}
In this example, we apply the results developed in the previous sections to a second-order inertial Newton-like system, as studied in \cite{Attouch} and further analyzed in \cite{dao2023locating}.

Let $\Phi : \R^n \to \R$ be a twice continuously differentiable function, bounded from below, whose Hessian $\nabla^2 \Phi$ is Lipschitz continuous on bounded subsets of $\R^n$. Consider the second-order dynamical system
\begin{equation}
\label{DIN}
\ddot{x}(t) + \alpha \dot{x}(t) + \nabla \Phi(x(t)) + \beta \nabla^2 \Phi(x(t)) \dot{x}(t) = 0,
\end{equation}
where $\alpha > 0$ and $\beta > 0$ are fixed parameters.
It is shown in \cite[Example 4.7]{dao2023locating} that every solution $x(\cdot)$ of \eqref{DIN} satisfies
\[
\lim_{t \to +\infty} \mathbf{d}(x(t); \mathcal{N}) = 0
\quad \text{and} \quad
\lim_{t \to +\infty} \dot{x}(t) = 0,
\]
where $\mathcal{N} := \{x \in \R^n : \nabla \Phi(x) = 0\}$ is the set of critical points of $\Phi$. In particular, if $\Phi$ is convex, then $\mathcal{N} = \argmin \Phi$.
We reformulate \eqref{DIN} as a first-order system by introducing the state variable $y(t) := (x(t), \dot{x}(t)) \in \R^{2n}$, leading to
\[
\dot{y}(t) \in f(y(t)) - A(y(t)),
\]
where
\[
f(y) := (y_2, 0),
\quad
A(y) := \left(0, \nabla \Phi(y_1) + \beta \nabla^2 \Phi(y_1) y_2\right).
\]
Here, the domain is $C = \R^{2n}$, so the normal cone is trivial.
The operator $A$ is single-valued and continuous. Consequently, for every $y \in \R^{2n}$, we have
\[
\clco A_0(y) = \{ A(y) \}.
\]
Thus, $\clco(A_0)$ is also single-valued and continuous.
Along the bounded trajectories $(x(t), \dot{x}(t))$, the mapping $\clco(A_0)$ is Lipschitz continuous and uniformly bounded. The boundedness of solutions follows from the fact that the modified energy function
\[
V(y(t)) := (\alpha \beta + 1) \Phi(x(t)) + \frac{1}{2} \left\| \dot{x}(t) + \beta \nabla \Phi(x(t)) \right\|^2
\]
is nonincreasing over time and bounded from below, as established in \cite{dao2023locating}. The associated function $W : \R^{2n} \to [0,+\infty)$ is given by
\[
W(y) := \alpha \|\dot{x}\|^2,
\quad \text{for } y = (x, \dot{x}).
\]
Thus, the pair $(V, W)$ satisfies the strict Hamiltonian decrease condition~\eqref{eq:H2}. We now apply our theoretical results. We are considering the set of equilibria
\[
\mathcal{N} = \{x \in \R^n : \nabla \Phi(x) = 0\}.
\]
Moreover, we have $W^{-1}(0) = \mathcal{N}$, and every point of $\mathcal{N}$ is Lyapunov stable. These properties match exactly the assumptions required to apply our main result, Proposition~\ref{ConvW}. Thus, by Proposition~\ref{ConvW}, we conclude that the set $\mathcal{N}$ is pointwise asymptotically stable $(\mathbf{PAS})$  for the system \eqref{DIN}.
Furthermore, by Theorem~\ref{SS}, we deduce that every solution $x(\cdot)$ converges to a point $z \in \mathcal{N}$. In particular, if $\Phi$ is convex, then $\mathcal{N} = \argmin \Phi$, and $\displaystyle\lim_{t \to +\infty} x(t) = z$ for some $z \in \argmin \Phi$.
\end{example}
\begin{example}[Nonsmooth differential inclusion with convex potential]
\label{ex2order}
In this example, we apply our results to a second-order differential inclusion involving a nonsmooth potential, as studied in \cite{saoud15} in the context of unilateral mechanics.

Let $\Phi : \R^n \to \R \cup \{+\infty\}$ be a proper, convex, and lower semicontinuous function. Given an initial condition $(x_0, \dot{x}_0) \in \dom(\Phi) \times \dom(\Phi)$ with $(x_0, \dot{x}_0) = (x(0), \dot{x}(0))$, and fixed parameters $m > 0$, $\alpha > 0$, and $\beta > 0$, consider the second-order differential inclusion
\begin{equation}
\label{secorder}
m \ddot{x}(t) + \alpha \dot{x}(t) + \beta x(t) \in -\partial \Phi(\dot{x}(t)),
\quad \text{for a.e. } t \geq 0.
\end{equation}
Introducing the state variable $y(t) := (x(t), \dot{x}(t)) \in \R^{2n}$, we rewrite \eqref{secorder} as a first-order system:
\[
\dot{y}(t) \in f(y(t)) - A(y(t)),
\]
where
\[
f(y) := 
\begin{bmatrix}
0 & 1 \\
-\frac{\beta}{m} & -\frac{\alpha}{m}
\end{bmatrix}
y,
\quad \text{and} \quad
A(y) :=
\left\{
\begin{bmatrix}
0 \\
\frac{1}{m} \xi
\end{bmatrix}
: \xi \in \partial \Phi(y_2)
\right\}.
\]
On the subset $E$ where $\partial \Phi$ is single-valued (e.g., inside $\operatorname{int}(\dom(\Phi))$), we define the mapping $A_0 : \R^{2n} \to \R^{2n}$ by
\[
A_0(y) := 
\begin{bmatrix}
0 \\
\frac{1}{m} s(y_2)
\end{bmatrix},
\]
where $s$ is a single-valued selection from $\partial \Phi$ on $E$.
Since $\partial \Phi$ is closed and convex, and the normal cone vanishes, we have $A = \operatorname{clco}(A_0)$ on $\R^{2n}$.

We assume that \(A\) is uniformly bounded on bounded subsets. Otherwise, a Lipschitz continuous approximation \(F_k\) and a Lipschitz selection \(\psi_k\) can be introduced as described in Section~\ref{sec:selection}. The set of equilibrium points of \eqref{secorder} is
\[
\mathcal{E} := \left\{ (\bar{y}_1, 0) \in \R^{2n} : \bar{y}_1 \in -\frac{1}{\beta} \partial \Phi(0) \right\},
\]
where $\partial \Phi(0)$ is assumed to be nonempty.
For each $\bar{y}_1 \in \mathcal{N} := -\frac{1}{\beta} \partial \Phi(0)$, consider the Lyapunov function
\[
V(y) := \frac{\beta}{2m} \| y_1 - \bar{y}_1 \|^2 + \frac{1}{2} \| y_2 \|^2.
\]
Along any solution $y(t)$, the derivative satisfies
\[
\dot{V}(y(t)) \leq \max_{v \in A(y(t))} \langle \nabla V(y(t)), f(y(t)) - v \rangle \leq 0,
\]
so that \(V\) is nonincreasing along trajectories. As a result, we obtain
\[
\lim_{t \to +\infty} \mathbf{d}(x(t); \mathcal{N}) = 0
\quad \text{and} \quad
\lim_{t \to +\infty} \dot{x}(t) = 0,
\]
which shows that every point of \(\mathcal{E}\) is Lyapunov stable.
Since the pair \((V, W)\), with \(W \equiv 0\), satisfies the non-strict Hamiltonian decrease condition~\eqref{eq:H1}, and every point of \(\mathcal{E}\) is Lyapunov stable, we apply Theorem~\ref{SS} to conclude that the system \eqref{secorder} is semistable  $(\mathbf{SS})$.

To illustrate, let us consider the case where \(\Phi(x) = \|x\|\) on \(\R^n\).  As shown in Example~\ref{ex:approx}, the mapping \(\clco(A_0)\) coincides with the Clarke subdifferential \(\partial_C \Phi\) and is uniformly bounded on bounded subsets of \(\R^n\). Moreover, a Lipschitz approximation \(F_k\) and a Lipschitz selection \(\psi_k\) have been explicitly constructed.
In this setting, the Lyapunov function \(V\) and the associated function \(W\) satisfy the strict Hamiltonian decrease condition~\eqref{eq:H2}, with
\[
W(y) = \frac{\alpha}{m} \| y_2 \|^2,
\quad \text{for } y = (y_1, y_2).
\]
The set of equilibrium points is
\[
\mathcal{E} = \left\{ (\bar{x}, 0) \in \R^n \times \R^n : \|\beta \bar{x}\| \leq 1 \right\},
\quad \text{with } \mathcal{E} \subseteq W^{-1}(0).
\]
Applying Proposition~\ref{ConvW}, we deduce that \(\mathcal{E}\) is pointwise asymptotically stable $(\mathbf{PAS})$ . Moreover, by Theorem~\ref{SS}, the second-order system \eqref{secorder} is semistable  $(\mathbf{SS})$. Thus, every solution \(x(\cdot)\) satisfies
\[
\lim_{t \to +\infty} \mathbf{d}\left(x(t); \left\{ \bar{x} \in \R^n : \|\beta \bar{x}\| \leq 1 \right\} \right) = 0
\quad \text{and} \quad
\lim_{t \to +\infty} \dot{x}(t) = 0.
\]
This example highlights how the most delicate assumptions of the framework can be verified in a nonsmooth setting. In particular, the uniform boundedness of $\clco A_0$, the construction of Lipschitz outer approximations, and the associated single-valued selections are made explicit for the choice $\Phi(x)=\|x\|$, as detailed in Example~\ref{ex:approx} and~\ref{appendixB}. This shows that the Hamiltonian inequality and the stability conditions can be checked concretely in nontrivial differential inclusions with noncompact equilibrium sets.
\end{example}
\section{Conclusion}
\label{sec:conclusion}
We have developed a structural approach to study the stability properties of differential inclusions governed by  maximally  monotone operators. By decomposing the operator into a convexified single-valued part and a normal cone, and combining this with Lipschitz regularizations and selection techniques, we have established sufficient conditions for pointwise asymptotic stability $(\mathbf{PAS})$  and semistability $(\mathbf{SS})$  without imposing strong assumptions on the system data. Our results extend the scope of Lyapunov analysis to broader classes of nonsmooth systems, highlighting the geometric structure underlying the dynamics. Several examples demonstrate the flexibility and robustness of the proposed framework.
\appendix
\renewcommand{\thesection}{Appendix \Alph{section}}
\section{The Steiner Point and Its Lipschitz Selection Property}
\label{appendixA}
For nonempty closed and bounded subsets \(A,B \subseteq \R^n\), the Hausdorff--Pompeiu distance is defined by
\[
\mathbf{d}_H(A,B)
:= \max\bigl\{ \sup_{u \in A} \mathbf{d}(u,B), \; \sup_{v \in B} \mathbf{d}(v,A) \bigr\},
\]
where \(\mathbf{d}(x,B) := \inf_{y \in B} \|x-y\|\) denotes the Euclidean distance. Let \(C \subset \mathbb{R}^n\) be a nonempty subset. The \emph{support function} of \(C\) is defined by
\[
\sigma_C(p) := \sup_{x \in C} \langle p, x \rangle,
\qquad p \in \mathbb{R}^n,
\]
with values in \((-\infty,+\infty]\).
If we assume that \(C\) is closed and convex, then  \(\sigma_C\) is a proper convex function. Moreover, if \(C\) is compact, then  for any \(p \in \mathbb{R}^n\)  the supremum is attained  and  the  convex subdifferential $\partial \sigma_C(p)$ of $ \sigma_C$ at $p$  is a nonempty compact convex subset of  $\mathbb{R}^n$  which satisfies:
\[
\partial \sigma_C(p)
= \operatorname*{arg\,max}_{x \in C} \langle p, x \rangle
= \{\, x \in C \mid \langle p, x \rangle = \sigma_C(p) \,\}.
\]
Consequently, \(\partial \sigma_C(p)\) admits a unique element of minimal norm, which we denote by
\[
(\partial \sigma_C(p))^\circ
:= \operatorname{proj}_{\partial \sigma_C(p)}(0).
\]
\begin{theorem}[{\cite[Theorem~9.4.1]{AubinFrankowska}}]
\label{thm:steiner-aubin}
Let \(\mathcal{K}\) denote the family of all nonempty convex compact subsets of \(\R^n\).
The \emph{Steiner point} of \(C \in \mathcal{K}\) is defined by
\[
s(C)
:= \frac{1}{\vol(\B)} \int_{\B} (\partial \sigma_C (p))^\circ \, dp,
\]
where \(\B \subseteq \R^n\) is the unit ball and \(\vol(\B)\) denotes its Lebesgue measure.
Then the following properties hold:
\begin{enumerate}
\item \(s(C) \in C\) for all \(C \in \mathcal{K}\);
\item Moreover, the mapping \(s : \mathcal{K} \to \R^n\) is Lipschitz continuous with respect
to the Hausdorff--Pompeiu distance: there exists a constant \(L>0\), depending only on the dimension \(n\),  such that
\[
\| s(C_1) - s(C_2) \|
\le L\, \mathbf{d}_H(C_1,C_2)
\qquad \forall\, C_1, C_2 \in \mathcal{K}.
\]
\end{enumerate}
\end{theorem}
For further information on the Steiner selection point and its properties, we refer the reader to \cite{Dentcheva1998,SaintPierre1985}.
\begin{lemma}[Steiner selection for Hausdorff--Lipschitz multifunctions]
\label{lem:steiner-selection}
Let \(G:\mathbb{R}^n \rightrightarrows \mathbb{R}^n\) be a set-valued mapping with nonempty, convex and compact values. Assume that \(G\) is locally Lipschitz continuous around \(\bar x \in \mathbb{R}^n\) with respect to the Hausdorff--Pompeiu distance, that is, there exist \(\ell>0\) and a neighborhood \(U\) of \(\bar x\) such that
\[
\mathbf{d}_H\bigl(G(x),G(x')\bigr)\le \ell \|x-x'\|
\qquad \forall\, x,x'\in U.
\]
Define the Steiner selection \(\psi:U\to\mathbb{R}^n\) by
\[
\psi(x):=s(G(x)), \qquad x\in U.
\]
Then \(\psi\) is Lipschitz continuous on \(U\). More precisely, there exists a constant
\(L>0\) such that
\[
\|\psi(x)-\psi(x')\|
\le L\,\|x-x'\|
\qquad \forall\, x,x'\in U,
\]
where \(L>0\) depends only on the dimension \(n\) and the Hausdorff--Lipschitz constant
\(\ell\) of \(G\).
\end{lemma}
\begin{proof}
Let \(x,x'\in U\). By Theorem~\ref{thm:steiner-aubin}, there exists a constant \(L_n>0\), depending only on the dimension \(n\), such that
\[
\|s(G(x))-s(G(x'))\|\le L_n\, \mathbf{d}_H(G(x),G(x')).
\]
Using the Hausdorff--Lipschitz continuity of \(G\) on \(U\), we obtain
\[
\|\psi(x)-\psi(x')\|
= \|s(G(x))-s(G(x'))\|
\le L_n\, \mathbf{d}_H(G(x),G(x'))
\le L_n\,\ell\,\|x-x'\|.
\]
Setting \(L:=L_n\,\ell\) yields
\[
\|\psi(x)-\psi(x')\|\le L\,\|x-x'\|
\qquad \forall\, x,x'\in U,
\]
which proves that \(\psi\) is Lipschitz continuous on \(U\).
\end{proof}
\begin{example}[Steiner point of the unit ball]
\label{ex:steiner-ball}
Let $\B \subset \R^n$ denote the closed unit ball.
We compute explicitly the Steiner point $s(\B)$ using the representation
given in Theorem~\ref{thm:steiner-aubin}. Recall that the support function of $\B$ is
\[
\sigma_{\B}(p) = \sup_{\|x\|\le 1} \langle p,x\rangle = \|p\|,
\qquad p\in\R^n.
\]
For $p\neq 0$, the subdifferential of $\sigma(\B,\cdot)$ at $p$ is the singleton
\[
\partial \sigma_{\B}(p) = \left\{ \frac{p}{\|p\|} \right\},
\]
and therefore
\[
(\partial \sigma_{\B}(p))^\circ = \frac{p}{\|p\|}.
\]
At $p=0$, one has $\partial \sigma(\B,0)=\B$, whose element of minimal norm is $0$. Using the definition of the Steiner point, we obtain
\[
s(\B)
= \frac{1}{\vol(\B)} \int_{\B} (\partial \sigma_{\B}(p))^\circ \, dp
= \frac{1}{\vol(\B)} \int_{\B} \frac{p}{\|p\|}\, dp.
\]
The mapping $p \mapsto p/\|p\|$ is odd, and the unit ball $\B$ is symmetric with respect to the origin. Hence the integral vanishes, yielding
\[
s(\B)=0.
\]
This example illustrates how the Steiner point selects the center of symmetry for centrally symmetric convex bodies, in contrast with projection-based selections which are defined via minimal-norm criteria.
\end{example}

\section{Detailed Analysis of the Lipschitz Approximation in Example~\ref{ex:approx}}\label{appendixB}
Let \( C \subseteq \R^n \) be a nonempty closed convex set, and consider the distance function
\[
\varphi(x) := \mathbf{d}(x, C) = \inf_{y \in C} \|x - y\|.
\]
This function is convex and globally Lipschitz with constant 1. Its Clarke subdifferential is given by
\[
\partial_C \varphi(x) =
\begin{cases}
\left\{ u(x) \right\}, & \text{if } x \notin C, \\[5pt]
N_C(x) \cap \B, & \text{if } x \in C,
\end{cases}
\]
where \( v(x) := x - \proj_C(x)\) and \(u(x) := \dfrac{v(x)}{\|v(x)\|}\in\mathbb{S}^{n-1}\).

Our goal is to approximate this subdifferential with a family of Lipschitz continuous, convex-valued, and uniformly bounded set-valued mappings 
\( F_k : \R^n \rightrightarrows \R^n \), indexed by a parameter \( \delta_k := 1/k \).

We construct the approximation \( F_k : \R^n \rightrightarrows \R^n \) as:
\[
F_k(x) :=
\begin{cases}
\left\{ u(x) \right\}, & \text{if } \|v(x)\| \ge \delta_k, \\[8pt]
\left(1 - \alpha(x) \right)\B + \alpha(x) \left\{ u(x) \right\}, & \text{if } 0 < \|v(x)\| < \delta_k, \\[8pt]
N_C(x) \cap \B, & \text{if } x \in C.
\end{cases}
\]
with \(\alpha(x) := \dfrac{\|v(x)\|}{\delta_k} \in (0,1)\).

We show that \( F_k(x) \subseteq \B \) for all \( x \in \R^n \). Consider the cases:
\begin{itemize}
    \item If \( \|v(x)\| \ge \delta_k \): then \( F_k(x) = \{ u(x) \} \), a unit vector on the sphere \( \partial \B \), so \( F_k(x) \subseteq \B \).
    \item If \( 0 < \|v(x)\| < \delta_k \): the set \( F_k(x) \) is a convex combination of \( \B \) and a point on \( \partial \B \). Since both are in \( \B \), so is the combination.
    \item If \( x \in C \): then \( v(x) = 0 \), and \( F_k(x) = N_C(x) \cap \B \subseteq \B \) by definition.
\end{itemize}
Thus, \( F_k(x) \subseteq \B \) for all \( x \), proving uniform boundedness. We show that \( x \mapsto F_k(x) \) is Lipschitz continuous under the Hausdorff-Pompeiu distance. Since \( \proj_C \) is non-expansive, we have:
\[
\|v(x) - v(y)\| \le \|x - y\| + \|\proj_C(x) - \proj_C(y)\| \le 2 \|x - y\|.
\]
So \( v(x) \) is Lipschitz with constant 2. Consider now the different regions.
\begin{itemize}
    \item \textbf{Case: \(\|v(x)\|, \|v(y)\| \ge \delta_k\).}

In this region, both \(F_k(x)\) and \(F_k(y)\) are singleton sets:
\[
F_k(x) = \left\{ u(x) \right\}, \quad F_k(y) = \left\{ u(y) \right\},
\]
Since these are singleton sets, the Hausdorff-Pompeiu distance between \(F_k(x)\) and \(F_k(y)\) reduces to the Euclidean distance between the points:
\[
\mathbf{d}_H(F_k(x), F_k(y)) = \|u(x) - u(y)\|.
\]
To show Lipschitz continuity, we analyze the mapping \(x \mapsto u(x)\). First, recall that the projection onto a closed convex set is non-expansive, i.e.,
\[
\|\proj_C(x) - \proj_C(y)\| \le \|x - y\|.
\]
Thus, the mapping \(x \mapsto v(x) := x - \proj_C(x)\) is Lipschitz with constant 2:
\[
\|v(x) - v(y)\| \le \|x - y\| + \|\proj_C(x) - \proj_C(y)\| \le 2\|x - y\|.
\]
Now consider the mapping \(v \mapsto \dfrac{v}{\|v\|}\), which is smooth on the set \(\{v \in \R^n : \|v\| \ge \delta_k\}\). On this domain, it is Lipschitz continuous with constant depending on \(\delta_k\). Specifically, for any vectors \(v, w \in \R^n\) such that \(\|v\|, \|w\| \ge \delta_k\), we have
\[
\begin{aligned}
\left\| \frac{v}{\|v\|} - \frac{w}{\|w\|} \right\|
&= \left\| \frac{v - w}{\|v\|} + w \left( \frac{1}{\|v\|} - \frac{1}{\|w\|} \right) \right\| \\[5pt]
&\le \left\| \frac{v - w}{\|v\|} \right\| + \left\| w \left( \frac{1}{\|v\|} - \frac{1}{\|w\|} \right) \right\| \\[5pt]
&\le \frac{\|v - w\|}{\|v\|} + \frac{\|w\|}{\|v\| \|w\|} \, \left| \|w\| - \|v\| \right| \\[5pt]
&\le \frac{\|v - w\|}{\delta_k} + \frac{1}{\delta_k} \, \|v - w\| \\[5pt]
&= \frac{2}{\delta_k} \|v - w\|.
\end{aligned}
\]
Applying this to \(v(x)\) and \(v(y)\), we obtain:
\[
\|u(x) - u(y)\| = \left\| \frac{v(x)}{\|v(x)\|} - \frac{v(y)}{\|v(y)\|} \right\| \le \frac{2}{\delta_k} \|v(x) - v(y)\| \le \frac{4}{\delta_k} \|x - y\|.
\]
Therefore,
\[
\mathbf{d}_H(F_k(x), F_k(y)) \le \frac{4}{\delta_k} \|x - y\|,
\]
so the mapping \(x \mapsto F_k(x)\) is Lipschitz continuous in this region.

\item \textbf{Case: \( 0 < \|v(x)\|, \|v(y)\| < \delta_k \).}

In this region, both \( F_k(x) \) and \( F_k(y) \) are closed balls with center \( \alpha(x) u(x) \) and radius \( 1 - \alpha(x) \).  Note that both \( \alpha(x) \) and \( u(x) \) are Lipschitz continuous functions on this region, Since the mapping \(x \mapsto v(x)\) is Lipschitz continuous, the above estimates
are performed on regions where \(\|v(x)\|\) is uniformly controlled. Then, for some constants \( L_\alpha, L_u > 0 \), we have:
\[
|\alpha(x) - \alpha(y)| \le L_\alpha \|x - y\|, \quad \text{and} \quad \|u(x) - u(y)\| \le L_u \|x - y\|.
\]
The distance between the centers of the two balls satisfies:
\begin{align*}
\|\alpha(x) u(x) - \alpha(y) u(y)\| 
&\le |\alpha(x) - \alpha(y)| \cdot \|u(x)\| + \alpha(y) \cdot \|u(x) - u(y)\| \\
&\le |\alpha(x) - \alpha(y)| + \|u(x) - u(y)\| \\
&\le (L_\alpha + L_u) \|x - y\|.
\end{align*}
The difference between the radii is:
\[
|1 - \alpha(x) - (1 - \alpha(y))| = |\alpha(x) - \alpha(y)| \le L_\alpha \|x - y\|.
\]
Hence, the Hausdorff-Pompeiu distance between \( F_k(x) \) and \( F_k(y) \) is bounded above by the sum of these two quantities:
\[
\mathbf{d}_H(F_k(x), F_k(y)) \le (2L_\alpha + L_u) \|x - y\|.
\]
This shows that \( F_k \) is Lipschitz continuous under the Hausdorff-Pompeiu distance in this region.
\end{itemize}
We now analyze the behavior of \( F_k(x) \) near the transition zones and prove that it varies continuously in the Hausdorff-Pompeiu sense.
\begin{itemize}
    \item As \( \|v(x)\| \to \delta_k^{-} \):
Suppose \( \|v(x)\| < \delta_k \) but close to \( \delta_k \). Then \( F_k(x) \) is given by:
\[
F_k(x) = (1 - \alpha(x)) \B + \alpha(x) \left\{ u(x) \right\}.
\]
As \( \|v(x)\| \to \delta_k^{-} \), we have \( \alpha(x) \to 1 \). Consequently:
\begin{itemize}
    \item The weight on \( \B \), namely \( 1 - \alpha(x) \), tends to 0.
    \item The weight on the unit vector tends to 1.
\end{itemize}
Therefore, the convex combination \( F_k(x) \) collapses to:
\[
F_k(x) \to \left\{ u(x) \right\}.
\]
This matches the definition of \( F_k(x) \) for \( \|v(x)\| \ge \delta_k \), ensuring continuity at the threshold \( \|v(x)\| = \delta_k \).
\item  As \( \|v(x)\| \to 0 \) (i.e., \( x \to C \)):
Within the region \( 0 < \|v(x)\| < \delta_k \), the expression is again:
\[
F_k(x) = (1 - \alpha(x)) \B + \alpha(x) \left\{ u(x) \right\}. 
\]
As \( \|v(x)\| \to 0 \), we get \( \alpha(x) \to 0 \). Hence:
\begin{itemize}
    \item The center of the ball, \( \alpha(x) \cdot u(x) \), tends to the origin.
    \item The radius \( 1 - \alpha(x) \to 1 \).
\end{itemize}
Thus, the ball tends to the unit ball:
\[
F_k(x) \to \B.
\]
\item  At \( x \in C \):
We define:
\[
F_k(x) := N_C(x) \cap \B.
\]
The normal cone mapping \( x \mapsto N_C(x) \) is outer semicontinuous, and \( \B \) is compact. Hence, the composition \( x \mapsto F_k(x) \) is outer semicontinuous. Therefore:
\[
\limsup_{x \to x_0 \in C} F_k(x) \subseteq F_k(x_0),
\]
so no discontinuity occurs at the boundary \( x \in C \).
\end{itemize}
Finally, we study the convergence to \( \partial_C \varphi \). If \( x \notin C \), then for large \( k \),  \( \|v(x)\| \ge \delta_k \), so
\[
F_k(x) = \left\{ u(x) \right\} = \partial_C \varphi(x).
\]
If \( x \in C \), then \( v(x) = 0 \), so
\[
F_k(x) = N_C(x) \cap \B = \partial_C \varphi(x).
\]
If \( x \to x_0 \in C \), then \( \|v(x)\| \to 0 \), so \( \alpha(x) \to 0 \), so
\[
F_k(x) = (1 - \alpha(x))\B + \alpha(x) \left\{ u(x) \right\} \to \B.
\]
and by outer semicontinuity of the normal cone 
\[
\limsup_{x \to x_0} F_k(x) \subseteq N_C(x_0) \cap \B = \partial_C \varphi(x_0).
\]
which proves that \( F_k \to \partial_C \varphi \) graphically.

We now define the selection associated with the approximation \(F_k\) via
the Steiner point:
\[
\psi_k(x) := s\bigl(F_k(x)\bigr), \qquad x \in \R^n,
\]
where \(s\) denotes the Steiner point mapping recalled in~\ref{appendixA}.
To compute the Steiner selection \(\psi_k(x)=s(F_k(x))\), we distinguish the
three regions in the definition of \(F_k(x)\).

\textbf{Case 1: \(\|v(x)\|\ge \delta_k\).}
In this region,
\[
F_k(x)=\{u(x)\}.
\]
Since the Steiner point satisfies \(s(\{a\})=a\) for every \(a\in\R^n\),
we obtain
\[
\psi_k(x)=s(F_k(x))=s(\{u(x)\})=u(x).
\]

\textbf{Case 2: \(0<\|v(x)\|<\delta_k\).}
Here,
\[
F_k(x)=(1-\alpha(x))\B+\alpha(x)\{u(x)\}
=\alpha(x)u(x)+(1-\alpha(x))\B.
\]
Using the translation covariance and positive homogeneity of the Steiner
point,
\[
s(a+K)=a+s(K),
\qquad
s(\lambda K)=\lambda s(K)\quad(\lambda\ge 0),
\]
we get
\[
\psi_k(x)=s(F_k(x))
=s\bigl(\alpha(x)u(x)+(1-\alpha(x))\B\bigr)
=\alpha(x)u(x)+(1-\alpha(x))\,s(\B).
\]
Finally, since \(\B\) is centrally symmetric, \(s(\B)=0\) (see Example~\ref{ex:steiner-ball}), hence
\[
\psi_k(x)=\alpha(x)u(x).
\]

\textbf{Case 3: \(x\in C\).}
Then
\[
F_k(x)=N_C(x)\cap \B,
\qquad\text{and}\qquad
\psi_k(x)=s\bigl(N_C(x)\cap \B\bigr).
\]
In general, there is no closed-form expression for \(s(N_C(x)\cap\B)\) without additional structure on \(C\) (or on the cone \(N_C(x)\)). In particular, one cannot expect \(\psi_k(x)=0\) on \(C\) in general.
Altogether, the Steiner selection admits the piecewise representation
\[
\psi_k(x)=s(F_k(x))=
\begin{cases}
u(x), & \text{if } \|v(x)\|\ge \delta_k,\\[6pt]
\alpha(x)\,u(x), & \text{if } 0<\|v(x)\|<\delta_k,\\[6pt]
s\bigl(N_C(x)\cap \B\bigr), & \text{if } x\in C.
\end{cases}
\]
This function \( \psi_k \) is globally defined, Lipschitz continuous, and uniformly bounded by 1. We split the analysis into three cases: 
\begin{itemize}
    \item \textbf{Case 1: \(\|v(x)\| \ge \delta_k\) and \(\|v(y)\| \ge \delta_k\).} We have
\[
\begin{aligned}
\left\| \psi_k(x) - \psi_k(y) \right\|
&= \left\|u(x) - u(y) \right\|\\[5pt]
&= \left\| \frac{v(x)}{\|v(x)\|} - \frac{v(y)}{\|v(y)\|} \right\| \\[5pt]
&= \left\| \frac{v(x)}{\|v(x)\|} - \frac{v(y)}{\|v(x)\|} + \frac{v(y)}{\|v(x)\|} - \frac{v(y)}{\|v(y)\|} \right\| \\[5pt]
&\le \left\| \frac{v(x) - v(y)}{\|v(x)\|} \right\| + \left\| v(y) \left( \frac{1}{\|v(x)\|} - \frac{1}{\|v(y)\|} \right) \right\| \\[5pt]
&= \frac{\|v(x) - v(y)\|}{\|v(x)\|} + \|v(y)\| \cdot \left| \frac{1}{\|v(x)\|} - \frac{1}{\|v(y)\|} \right| \\[5pt]
&\le \frac{\|v(x) - v(y)\|}{\delta_k} + \frac{1}{\delta_k} \left| \|v(x)\| - \|v(y)\| \right| \\[5pt]
&\le \frac{2}{\delta_k} \|v(x) - v(y)\| \\[5pt]
&\le \frac{4}{\delta_k} \|x - y\|.
\end{aligned}
\]
\item \textbf{Case 2: \(0 < \|v(x)\| < \delta_k\) and \(0 < \|v(y)\| < \delta_k\).}
We have \(\psi_k(x)=\alpha(x)u(x)\) with
\(\alpha(x)=\|v(x)\|/\delta_k\).
Since \(v(x)\neq 0\) on this region, we observe that
\[
\psi_k(x)=\alpha(x)u(x)
=\frac{\|v(x)\|}{\delta_k}\,\frac{v(x)}{\|v(x)\|}
=\frac{1}{\delta_k}v(x),
\]
and similarly \(\psi_k(y)=\dfrac{1}{\delta_k}v(y)\).
Therefore,
\[
\begin{aligned}
\left\| \psi_k(x) - \psi_k(y) \right\|
&= \frac{1}{\delta_k}\,\|v(x)-v(y)\| \\[5pt]
&\le \frac{2}{\delta_k}\,\|x-y\|.
\end{aligned}
\]
\item \textbf{Case 3: If $x \in C$}. Since \(F_k(x)\subseteq \B\) for all \(x\in\R^n\) and the Steiner point is a
selector, we have
\[
\psi_k(x)=s\bigl(F_k(x)\bigr)\in F_k(x)\subseteq \B,
\]
and therefore
\[
\|\psi_k(x)\|\le 1
\qquad \forall\, x\in\R^n.
\]
The Lipschitz continuity of \(\psi_k\) follows from the Lipschitz continuity of \(F_k\) in the Hausdorff distance and that of the Steiner point; see~\ref{appendixA}.
\end{itemize}
\begin{remark}
In the present example, one may also define the selection
\[
\psi_k(x):=\proj_{F_k(x)}(0)
\]
and obtain a globally Lipschitz continuous mapping. This relies on the special geometry of the sets \(F_k(x)\), which yields the explicit representation
\[
\psi_k(x) =
\begin{cases}
u(x), & \text{if } \|v(x)\| \ge \delta_k, \\[8pt]
\beta(x)\,u(x), & \text{if } 0 < \|v(x)\| < \delta_k, \\[8pt]
0, & \text{if } x \in C,
\end{cases}
\qquad
\text{where }
\beta(x):=\dfrac{2\|v(x)\|}{\delta_k}-1.
\]
Indeed, away from \(C\), the sets \(F_k(x)\) are either singletons or Euclidean balls whose centers and radii depend Lipschitz continuously on \(x\), for which the projection of the origin admits an explicit and Lipschitz formula. Moreover, for \(x\in C\), the origin belongs to \(F_k(x)=N_C(x)\cap\B\), so the projection is identically zero.

This behavior is specific to this construction and should not be expected in general. Without additional geometric assumptions such as uniform strong convexity of the values \(F_k(x)\), the projection-based selection \(x\mapsto\proj_{F_k(x)}(0)\) is, in general, not Lipschitz continuous and is typically only Hölder continuous of order \(1/2\), (see~\cite{Thibaultbook}). For this reason, in the general framework developed in this paper, we rely on the Steiner selection, which is Lipschitz continuous with respect to the Hausdorff distance and does not require additional curvature assumptions on the values of \(F_k\).
\end{remark}



\bibliographystyle{plain}
\bibliography{refs}

\end{document}